\pdfoutput=1
\documentclass[notitlepage,nofootinbib,superscriptaddress,showkeys,showpacs,longbibliography,11pt]{article}

\usepackage{jheppub}

\usepackage{amsmath,amssymb,amsthm,latexsym,bbm,calc}
\usepackage[english]{babel}
\usepackage{graphicx,color}
\usepackage{xspace}
\usepackage{hyperref}
\usepackage{pifont,dsfont}
\usepackage{marvosym}
\usepackage{slashed}
\usepackage{subfigure}
\usepackage[dvipsnames]{xcolor}
\usepackage{float}

\usepackage{tikz}
\usetikzlibrary{decorations.markings}
\usetikzlibrary{arrows}
\usepackage{textcomp}
\def\lcross{{>\!\!\!\triangleleft}}
\def\rcocross{{\blacktriangleright\!\!<}}

\def\rlbicross{{\triangleright\!\!\!\blacktriangleleft}}
\def\lrbicross{{\blacktriangleright\!\!\!\triangleleft}}
\def\dcross{{\bowtie}}

\def\rbiprod{{\cdot\kern-.33em\triangleright\!\!\!<}}
\def\lbiprod{{>\!\!\!\triangleleft\kern-.33em\cdot}}

\def\dprod{{\cdot\kern-.55em\triangleright\!\!\triangleleft\kern-.33em\cdot}}

\def\ot{\otimes}

\newcommand{\C}{\mbox{${\mathbb C}$}}
\newcommand{\R}{\mbox{${\mathbb R}$}}
\newcommand{\N}{\mbox{${\mathbb N}$}}
\newcommand{\Z}{\mbox{${\mathbb Z}$}}

\newcommand{\CH}{\hbox{{$\mathcal H$}}}

\newcommand{\cP}{\mathcal P}

\newcommand{\mH}{\hbox{{$\mathfrak H$}}}
\newcommand{\id}{{\rm id}}

\newcommand{\eps}{{\epsilon}}
\newcommand{\tens}{\mathop{\otimes}}
\newcommand{\la}{{\triangleright}}
\newcommand{\ra}{{\triangleleft}}

\newcommand{\<}{\langle}
\renewcommand{\>}{\rangle}

\renewcommand{\o}{{}_{\scriptscriptstyle(1)}}
\renewcommand{\t}{{}_{\scriptscriptstyle(2)}}
\newcommand{\thr}{{}_{\scriptscriptstyle(3)}}
\newcommand{\fo}{{}_{\scriptscriptstyle(4)}}
\newcommand{\fiv}{{}_{\scriptscriptstyle(5)}}
\newcommand{\si}{{}_{\scriptscriptstyle(6)}}
\newcommand{\sev}{{}_{\scriptscriptstyle(7)}}
\newcommand{\ei}{{}_{\scriptscriptstyle(8)}}
\newcommand{\ten}{{}_{\scriptscriptstyle(10)}}
\newcommand{\ele}{{}_{\scriptscriptstyle(11)}}
\newcommand{\twe}{{}_{\scriptscriptstyle(12)}}
\newcommand{\thi}{{}_{\scriptscriptstyle(13)}}
\newcommand{\fou}{{}_{\scriptscriptstyle(14)}}
\newcommand{\fif}{{}_{\scriptscriptstyle(15)}}
\newcommand{\sixt}{{}_{\scriptscriptstyle(16)}}
\newcommand{\Nine}{{}_{\scriptscriptstyle(9)}}
\newcommand{\sevt}{{}_{\scriptscriptstyle(17)}}
\newcommand{\eit}{{}_{\scriptscriptstyle(18)}}

\newcommand{\sso}{{}^{\scriptscriptstyle(1)}}
\newcommand{\sst}{{}^{\scriptscriptstyle(2)}}

\renewcommand{\to}{{}_{\scriptscriptstyle(\tilde{1})}}
\renewcommand{\ot}{{}_{\scriptscriptstyle(\tilde{2})}}
\newcommand{\tthr}{{}_{\scriptscriptstyle(\tilde{3})}}
\newcommand{\tfo}{{}_{\scriptscriptstyle(\tilde{4})}}

\newtheorem{thm}{Theorem}[section]
\theoremstyle{definition} 
\newtheorem{defn}[thm]{Definition}

\newtheorem{lemma}[thm]{Lemma}

\newtheorem{pro}[thm]{Proposition}

\def\be{\begin{equation}}
\def\ee{\end{equation}}

\newcommand{\bes}{\begin{eqnarray}}
\newcommand{\ees}{\end{eqnarray}}

\hypersetup{
	colorlinks=true,       
	linkcolor=MidnightBlue,          
	citecolor=BrickRed,        
	filecolor=BrickRed,      
	urlcolor=BrickRed           
}

\begin{document}
\title{Semidual Kitaev  lattice model
 and tensor network representation}

\author[a]{{\bf Florian Girelli,}}\emailAdd{fgirelli@uwaterloo.ca}
\affiliation[a]{Department of Applied Mathematics, University of Waterloo, Waterloo, Ontario, Canada}

\author[b, c, d]{{\bf Prince K. Osei,}}\emailAdd{posei@quantumleapafrica.org}
\affiliation[b]{Quantum Leap Africa, AIMS Rwanda Center, Sector Remera, KN3 Kigali, Rwanda}
\affiliation[c]{Perimeter Institute for Theoretical Physics, 31 Caroline Street North, Waterloo, Ontario  N2L 2Y5, Canada}
\affiliation[d]{Department of Mathematics, University of Ghana, PO Box LG 25, Legon, Ghana}

\author[a,b]{{\bf Abdulmajid Osumanu}}\emailAdd{a3osuman@uwaterloo.ca}

\date{\small\today}

\medskip
\abstract{
Kitaev's lattice models 
  are usually defined as representations  of  the Drinfeld quantum double $D(H)=H\dcross H^{*\text{op}} $, as an example of a double cross product quantum group. We propose a new version based instead on $M(H)=H^{\text{cop}}\lrbicross H$ as an example of  Majid's bicrossproduct quantum group, related by semidualisation or  `quantum Born reciprocity' to $D(H)$.   
  Given a finite-dimensional Hopf algebra $H$, we show that a quadrangulated oriented surface defines a representation of the bicrossproduct quantum group $H^{\text{cop}}\lrbicross H$.
 Even though the bicrossproduct  has a more complicated and entangled coproduct, the construction of this new model is relatively natural as it relies on the use of the covariant Hopf algebra actions. Working locally, we obtain an exactly solvable Hamiltonian for the model and  provide a definition of the ground state  in terms of a tensor network representation.   }

\maketitle


\section{Introduction }

\subsection{Motivations }
The Kitaev quantum double models \cite{AYK} were originally proposed to exploit topological phases of matter for fault-tolerant quantum computation. The models are based on quantum many-body systems	 exhibiting topological order. Their physics is obtained from Topological quantum field theories (TQFTs), while their underlying  mathematical structure is based on Hopf algebras. For a given  finite group $G$, Kitaev constructed an `extended' Hilbert space on a triangulated oriented surface $\Sigma$ and an exactly solvable Hamiltonian, whose ground state or protected space is a topological invariant of the surface.  It turns out that  this triangulation or graph defines a representation of the Drinfeld quantum double  $D(G)$.  A well  known  example of these models is the Kitaev toric code, which is based on the cyclic group $\Z_2$ \cite{AYK}. See also \cite{Delgago2d3d} for a recent account. It was anticipated  in \cite{AYK} that these models could be generalized to be based  on a finite-dimensional Hopf algebra $H$. This was achieved in \cite{BMCA}. Other models in the family of topologically ordered spin models  such as the Levin-Wen string-net 
models \cite{Levin04, Levin06}  which are based on a representation category of $H$ are  also related to the Kitaev models \cite{Oliver09, Kadar08}. 
In particular, for a fusion category of representations of finite groups, a Fourier transformation of the Kitaev models lead to the extended string-net models \cite{Oliver09, Kadar09,Buerschaper2010}. The structure of excitations for these models is also well established \cite{KitaevKong,Lan14,Hu18}. One defines the so called ribbon operators on the Hilbert space that generate the excitations.

\smallskip
The Kitaev quantum double  models can be understood to describe the moduli space of flat connections on a 2d surface with defect excitations. 
From the point of view of quantum gravity, they are of strong interest as they are directly  related to certain 3d TQFTs defined in terms of (quasitriangular) Hopf algebras. 
It  is known  that the protected space of a Kitaev model for a finite-dimensional semisimple Hopf algebra $H$ on an oriented surface $\Sigma$  is exactly the vector space 
that the Turaev-Viro TQFTs \cite{BW96, Turaev92} for the representation category of $H$ assigns to $\Sigma$ \cite{Balsam12, Kadar09, Kadar08, Koenig10}.  
The construction of these models is also closely related to BF theory with defects \cite{Delcamp16,Clememt16,Bianca16,Horowitz89,JBaez95}, a TQFT describing locally flat connections.
Other recent examples include  a dual picture which  was introduced in the quantum gravity setting where the excitations have been swapped \cite{Dupuis:2017otn}. Even though this was discovered independently, this result could have been guessed in light of the  notion of electro-magnetic duality well known in topological quantum computing \cite{Buerschaper2010}. 
A recent paper by Meusburger  show that Kitaev's model for a finite-dimensional semisimple Hopf algebra $H$ is equivalent to the combinatorial quantization of Chern-Simons
theory for the Drinfeld double $D(H)$ \cite{CM}. This emerges in a gauge theoretic framework, in which both models are
viewed as Hopf algebra-valued lattice gauge theories \cite{MW}. 

These results have opened new perspectives on the relations between topological quantum information(TQI) and quantum gravity.
 Although each framework comes with its own motivation,  they share similar mathematical concepts.  For example, in the case of TQI cases, one deals with a (ribbon) graph decorated by Hopf algebra elements and constructs an exactly solvable Hamiltonian defined in terms of  operators acting on the nodes and faces of the graph. The vacuum state of this can be interpreted from the quantum gravity perspective as the pure gravity case, whereas the   excitations of the TQI Hamiltonian, used to perform quantum computations, are interpreted as particles with mass or spin depending on their location. In the case of  loop quantum gravity,  one has torsion excitations on the nodes, i.e. spin, whereas on the faces, one has curvature excitations, i.e. mass. The most relevant algebraic structure to deal with representations which classify particles for example and indicate their braiding, is not only the Hopf algebra $H$ but  the associated Drinfeld double $D(H)$. Once again, this  structure was identified using different arguments in each of the different frameworks. In the TQI case, one deals with the Drinfeld's quantum double of finite dimensional (semisimple) Hopf algebras (e.g. built from finite groups) \cite{AYK, BombinDeldado, BMCA} whereas in the quantum gravity case one makes use of the quantum double of Hopf algebras  built from Lie groups or their quantum deformation \cite{AGSI,AGSII,AS,Schroers,Freidel:2004vi,BNR,MeusburgerSchroers1,MeusburgerSchroers2,MN, noui2006three, SchroersReview}. 
\smallskip

As described above, the Drinfeld quantum double is in a sense the common quantum group which arise in the quantum computing setting. However, from the point of view of quantum gravity, other quantum groups emerge. In particular, the  bicrossproduct  quantum group originally proposed by Majid \cite{SM}  as a new foundation for quantum gravity.  The bicrossproduct quantum groups are interpreted here as algebras of observables of quantum systems so that one can view them as functions on a quantum phase space.
These bicrossproduct quantum groups are also known to be valid candidates for the combinatorial quantization of Chern-Simons theory of  3d gravity \cite{cath-bernd, OseiSchroers1, prince-bernd,PKO1}.
In particular, from the point of view of quantum group theory, the quantum double $D(H)=H\dcross H^{*\text{op}} $ is an example of a double cross product of Hopf algebras \cite{SM}.
  It turns out that these are related to 
  the bicrossproduct ones by Majid's idea of semidualisation or `quantum Born reciprocity',  
    proposed for quantum gravity where one can exchange position and momentum  degrees of freedom in an
algebraic framework \cite{Ma:pla}. 
The semidual partner of the quantum double here is the `mirror product' $M(H)=H^{\text{cop}}\lrbicross H$, so this is the natural candidate for a `semidual Kitaev model' which we propose here. 
It is also known that the  two  quantum groups are related  by a Drinfeld twist if $H$ is factorisable. This was  originally introduced in \cite{Majid94} as an algebraic Wick rotation, and recently  applied in   \cite{MO} to relate the bicrossproduct model quantum spacetime \cite{Majid:1994cy}  and the fuzzy $\R^3$ quantum spacetime by a module algebra twist as well as to find the  universal  $R$-matrix from the former model.

From the above considerations, while the bicrossproduct quantum groups emerge in a quantum gravity   framework, they are  yet to be explored for the topological quantum computation models. 
There is no general framework for the latter for general double crossproduct and nor at the moment will we find one for general bicrossproducts. However, the most important examble is the quantum double and we will see that it is possible to construct lattice quantum computation models for the case of the bicrossproduct corresponding to this. 

Our consideration in the present paper will be local, i.e effectively an infinite square lattice, with topological aspects needed to apply these results to the quadrangulation or other cellular decomposition of a general oriented connected surface $\Sigma$ to be considered elsewhere.

\smallskip

\subsection{General features of the  Kitaev model} 
We first recall the set up for the standard Kitaev model based on the double $D(H)$ but in a manner general enough to also apply to $M(H).$ In both cases the data we really need is a quantum group $D$ acting on an algebra $A$ (so the latter is a $D$-module algebra in quantum group parlance). { In 3D quantum gravity, $D$ could be the quantum Poinca\'e group acting on  quantum spacetime  $A$, which is then interpreted mathematically as a module algebra.}  We also need that $D$ factorises as an algebra into two subalgebras, $H_1,H_2$, which are each Hopf algebras   but not necessarily sub-Hopf algebras. Thus each element of $D$ can be written in the form $\sum_ia_ib_i$ for $a_i\in H_1$ and $b_i\in H_2$ uniquely in a certain tensor sense \cite{SM}.
The restriction of the action of $D$ to each subalgebra gives `triangle operators' $T_+^a, L^h_+$ respectively for their action on $A$.

It is convenient to suppose that $A$ is itself a Hopf algebra and we then define
\begin{equation}
T^{a}_{-} = S\circ T^{a}_{+}\circ S^{-1}, \quad 
 L^{h}_{-} = S\circ L^{h}_{+}\circ S^{-1}
\end{equation}
using its antipode.

{ In the case of the double $D(H)$, the triangle operators  $T_\pm^a, L^h_\pm$ make  $A$ respectively a $H_1$ module algebra and a $H_2$ module coalgebra. We choose to keep these features  as a general property of the Kitaev model. }

Now let $\Gamma$  be a graph embedded in an orientable Riemann surface $\Sigma$ and define a Hilbert space 
$H_\Gamma= A^{\tens |E|}$ where $E$ are the edges of the graph. On this we define operators $A^h(v,p), B^a(v,p)$ associated to each vertex $v$ and adjacent face $p$ as follows. First of all, the edges around $v$ and $p$ are both numbered in line with the orientation and prescribed starting point. If all the edges at $v$ point in then $A^h(v,p)$ is the tensor product action of $h \in D$ (which uses its coproduct and the $L_+$ on each such edge). If some of them point out then we use $L_-$ on those edges. 
Similarly, $B^a(v,p)$ is the tensor product action of $a \in D$ if all the adjacent edges to $p$ are clockwise (using the coproduct and $T_+$). If some of them are anticlockwise, we use $T_-$ on those. These conventions are depicted in Figure \ref{edge operators bis}. 

\begin{figure}[H]
\begin{center}
\includegraphics[scale=.85]{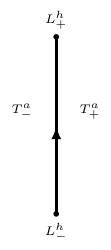}
\caption{Kitaev  convention for triangle operators acting on an edge. }
\label{edge operators bis}
\end{center}
\end{figure}

Other edges in $H_\Gamma$  are unaffected by $A^h(v, p)$ and $B^a(v, p).$ We will then let $a, h$ be integral elements to obtain `vertex and face' operators $A_v$ and $B_p$. Details will be given in Section 3.

The standard Kitaev model based on the quantum double has $D=D(H)=H\dcross H^{*\text{op}} $  acting on $A = H$, where $H$ is a finite dimensional Hopf algebra (with further properties). Here the two indicated factors are sub-Hopf algebras which makes things more straightforward. This `kinematic' level of
the data actually makes sense for any double cross product quantum group $H_1\dcross H_2$ in the sense of \cite{SM}, acting on $H_2^*$ . However, we can use the same set up for $D=M(H)=H^{\text{cop}}\lrbicross H$ acting on $A=H^*$, and this is what we will do. 
The big difference now is that $H$ here is not a sub-Hopf algebra but a quotient (there is a canonical Hopf algebra surjection $D \rightarrow H$), but it is still a subalgebra as part of an algebra cross product, so we can still
dissect the action of $D$ into $T_+,L_+$. We still define $T_-,L_-$ in the same way for the other orientations but we must be careful in the manner stated to
use the coproduct of $D$ when combining these to build $A^h(v, p)$ and $B^a(v, p).$
This much of the setup again makes sense for any bicrossproduct quantum group $H_1\lrbicross H_2$ as in \cite{SM} acting on $H_1^*$. In fact the data for these models are equivalent in the finite-dimensional case -- if you can build one then you can canonically build the other by a process of semidualisation \cite{SM, Ma:pla}.

In the standard case of $D(H)$, it can be shown that the associated Hamiltonian is topological, which means physically that the elementary excitations of the Kitaev model are anyons and facilitate the realisation of topological quantum computing by obeying some braid statistics \cite{AYK}. The braiding is possible as a result of the quasitriangular structure or universal $R$-matrix of $D(H)$. These aspects for the $M(H)$ case will be examined elsewhere but being isomorphic to
$H^{\text{cop}}\tens H$ as a Hopf algebra, one has a natural canonical quasitriangular structure if $H$ does. We will however, obtain an exactly solvable Hamiltonian. Following \cite{BMCA}, we then provide an explicit tensor network representation of the model. Such a representation is a starting point to explore some interesting physical properties of the system. For example in \cite{BMCA}, the tensor network representation is used to probe the notion of entanglement and check whether we have an area law for  entanglement entropy. It is also used to  define a notion of renormalization  implementing a hierarchy of states.

This completes our overview. Section 2 provides some preliminaries on Hopf algebras notations, the bicrossproduct quantum group $H^{\text{cop}}\lrbicross H$  and its action on $H^*$. In Section 3, we provide the detailed construction of the lattice representation based in this data and obtain the Hamiltonian for the model. In Section 4, we define the tensor network representation for our model and provide, in particular, the realisation of the ground state in our setting. We conclude in Section 5.

\section{Preliminaries}
\label{prelims}
We follow the notations and conventions from the book \cite{SM}.   Unless otherwise specified, we work over a field $k$ of characteristic zero.
A Hopf algebra or `quantum group' $H$ is an algebra and a coalgebra, with a linear coproduct $\Delta:H \longrightarrow  H\tens H$ which is an algebra homomorphism and satisfies the coassociativity condition $(\Delta\otimes \mbox{id})\circ \Delta=(\mbox{id}\otimes\Delta)\circ\Delta$.
We use  Sweedler notation for the coproduct so that for all $h\in H$, $\Delta (h)= h_{\o}\otimes h_{\t}=h\sso\otimes h\sst$.
There is also a counit $\eps:H \longrightarrow k$ and an antipode $S:H\longrightarrow H$ satisfying in particular $(Sh\o)h\t=h\o S h\t=\eps(h)$ for all $h\in H$. If $H$ is finite-dimensional, then $S^{-1}$ exist.
We denote by $H^{\tens n}$, $n\in \N$ the $n$-fold tensor product of $H$.  The composition of $n$ coproducts is the map $\Delta^{(n)}: H\rightarrow H^{\tens (n+1)}$  defined by $\Delta^{(n)}(h)=h_{\o} \tens h_{\t} \tens ... \tens h_{(n+1)}$. This is well defined since the coproduct is coassociative. 
We denote by $H^*$ the dual Hopf algebra with dual pairing given by the non-degenerate bilinear map $ \langle \,,\,\rangle$. $H^{\text{cop}}, H^{\text{op}}$ denote taking the opposite coproduct or opposite product in $H$.


An algebra $A$ is said to be an $H$-module \textit{algebra} if $A$  is a left $H$-module and this action is covariant,  i.e. 
\begin{equation}\label{module algb}
h\la (ab)= (h\o\la a)(h\t\la b),\hspace{1cm} h\la 1=\eps(h),\;\; a\in A, \, h\in H,
\end{equation}
where $\la$ denotes a left action.   
We say that  $(H,A)$ is a  \textit{covariant system}. 
We will generally prefer left actions as here but there is an analgous notion for a right module algebra by action $\ra$. This can always be coverted to a left one of the opposite algebra by
$h\triangleright a = a\triangleleft S^{-1}h$.  If $H$ acts on vector spaces $V,W$ then it also acts on $V\tens W $ by $h\la(v\tens w) = h\o\la v \tens h\t \la w$ for all $h\in H$, $v\in V$ and $w\in W$.

We now turn to the construction of  $M(H) =H^{\text{cop}}\lrbicross H$  from \cite{SM}. This is an example of Majid's theory of bicrossproducts $H_1\rlbicross H_2$ just as the more well-known $D(H) = H\dcross H^{*\text{op}}$  is an example of his theory of double cross products $H_1\dcross H_2$. We refer to \cite{SM} for details, while here suffice it to say that the key difference is that in the latter case each factor acts on the vector space of the other subject to certain axioms and the coproduct is the tensor product one. By contrast, in the bicrossproduct case the $H_2$ factor acts on the $H_1$ factor as a left $H_2$-module algebra and results in a semidirect or cross product algebra $H_1\lcross H_2$. Meanwhile the $H_1$ factor right coacts on the coalgebra of $H_2$ and results in a semidirect or cross coproduct coalgebra $H_1\rcocross H_2$ in Majid's notation. The construction on the coalgebra side here is conceptually dual to the construction on the algebra side. The bicrossproduct construction adds further axioms so that the two fit together to form a quantum group, with the merged notation.

In our case, $M(H)$ is an example as follows. 
The  left action of $H$ on $H^{\text{cop}}$ and the right coaction of $H^{\text{cop}}$ on $H$ are given respectively as 
\begin{equation}
\label{coaction of M(H)}
h\triangleright a =h_{\o}a Sh_{\t},\hspace{0.7cm} \Delta_{R}h=h_{\t}\otimes h_{\o}Sh_{\thr}.
\end{equation}
The algebra is 
\begin{eqnarray}
\label{sd4}
(a\otimes h)(b\otimes g) = a (h_{\o}b Sh_{\t})\otimes h_{\thr}g,\hspace{0.5cm}h,g\in H,\hspace{0.3cm}a,b\in H^{\text{cop}}.
\end{eqnarray}
Here, $H^{\text{cop}}\tens 1$ and $1\tens H$ appear as subalgebras but with mutual commutation relation  fully determined by 
\be
\label{strengthing mirror}
hb :=  (1\otimes h)(b\otimes 1)= (h_{\o}b Sh_{\t})h_{\thr},
\ee
where the identification  $h\rightarrow 1_{H^{\text{cop}}}\otimes h$ and $b\rightarrow b\otimes 1_{H}$ are algebra morphisms. The coproduct and antipode are respectively 
\begin{align}
\label{coproduct of M(H)}
\Delta(a\otimes h) =&\, a_{\t}\otimes h_{\t}\otimes a_{\o}h_{\o}Sh_{\thr}\otimes h_{\fo}, \\
\label{antipode of M(H)}
S(a\tens h) =&\, (1\tens Sh\t)(S(ah\o Sh\thr)\tens 1).
\end{align}
The Hopf algebra
$M(H)$ acts covariantly on $H^{*\text{op}}$ from the right according to 
\begin{equation}
\label{sd6}
\phi\triangleleft(a\otimes h)=\langle a h_{\o},\phi_{\o}\rangle\langle Sh_{\t},\phi_{\thr}\rangle \phi_{\t},
\end{equation}
and using  the antipode \eqref{antipode of M(H)} of $ M(H)$, this gives rise to covariant left action on $H^{*}$ 
\begin{equation}
\label{left covariant action on H*}
(a\otimes h)\triangleright \phi =\langle Sh_{\o}Sa ,\phi_{\o}\rangle\langle h_{\t},\phi_{\thr}\rangle \phi_{\t}.
\end{equation}
We refer to \cite{SM} and to the recent work \cite{MO} for more details.

\section{Semidual Kitaev model}
\label{BCPmodel}
In this section, we construct a lattice representation based on the mirror bicrossproduct  $H^{\text{cop}}\lrbicross H$ acting on 
$H^*$ and obtain an exactly solvable Hamiltonian for the model.  

Since we consider only the local quadrangulation of a $2$d oriented surface, we effectively work with $\Gamma$ a square lattice, without worrying about boundaries or the topology of the surface. We denote by $V,E,F $ respectively  the set of vertices, edges, faces of the graph $\Gamma$.
Given a finite-dimensional Hopf algebra $H$ with dual $H^*$, we define the 
 extended Hilbert  space $\mathcal{H}_{\Gamma}$ for the model by assigning $H^*$ to each edge of $\Gamma$ so that 
$$ \mathcal{H}_{{\Gamma}}=\bigotimes_{e\in \Gamma} H^{*},$$ the $|E|$-fold tensor product of  $H^*$ with each copy assigned to an edge of $\Gamma$. We identify $\phi\mapsto S(\phi)$, $\phi\in H^{*}$ if the orientation is reversed. Since $H^*$ is finite-dimensional, $S^2 = \id$ and this isomorphism is well defined. 


\subsection{Triangle operators }
To each edge $e\in E$, we assign a family of basic linear operators $(L^{h}_{\pm})_e$, $(T^{a}_{\pm})_e$ which are linear maps on the copy of $A=H^*$ in the Hilbert space $H_\Gamma = H^{*\tens |E|}$ associated to edge $e$, indexed by elements of the Hopf algebras $H$ and  $H^{\text{cop}}$ respectively.  They act trivially on the copies associated to other edges.
 These operators are called triangle operators \cite{AYK} and are defined as follows:
 
\begin{defn}\textit{ 
Let $H$ be a finite-dimensional Hopf algebra and $\Gamma$ a graph with cyclic ordering of edge ends at each vertex. Let  $h\in H$, $\phi\in H^{*}$ and $a\in H^{\text{cop}}$.  
The triangle operators for an edge $e\in E$ are linear maps 
\[
(L^{h}_{\pm})_e: H^{*\tens |E| }\rightarrow H^{*\tens |E|},\quad (T^{a}_{\pm})_e:H^{*\tens |E|}\rightarrow H^{*\tens |E|},
\]
where $L^h_{+}, T^a_+:H^{*} \rightarrow H^{*}$ acting on the copy associated to $e$ are given by}
\begin{eqnarray}
\label{triangle operators}
L^h_{+}(\phi) &=& \langle h,S\phi\o\phi\thr\rangle \phi\t, \hspace{0.7cm} L^h_{-}(\phi)=\langle h,\phi\thr S^{-1}\phi\o\rangle\phi\t,\nonumber\\
T^{a}_{+}(\phi) &=& \langle Sa, \phi\o\rangle \phi\t,  \hspace{1.8cm} T^{a}_{-}(\phi)=\langle a, \phi\t\rangle \phi\o.
\end{eqnarray}
\end{defn}
Here, the operators $L_+$ and $T_+$ are the restrictions to $1 \tens h$ and $a \tens1$ of the canonical left action  \eqref{left covariant action on H*}  of the bicrossproduct $ M(H)$ on $H^*$. The $L_-$ and $T_-$ are also left actions obtained using the relations 
\begin{equation}\label{relations}
{ L^{h}_{-}(\phi) = (S\circ L^{h}_{+}\circ S^{-1})(\phi), \quad T^{a}_{-}(\phi) = (S\circ T^{a}_{+}\circ S^{-1})(\phi).}
\end{equation}
Next, we denote by  $\rho^{a\tens h}_{+}= T^{a}_{+}L^{h}_{+}$ the full representation of $M(H)$ defined in \eqref{left covariant action on H*}. Using \eqref{relations} we have $\rho^{a\tens h}_{-}= T^{a}_{-}L^{h}_{-}$, i.e.,  
\be 
\label{rhominus}
\rho^{a\tens h}_{-}(\phi)=\langle a h\o,\phi\thr\rangle\langle h\t,S^{-1}\phi\o\rangle\phi\t.
\ee


\medskip

It is interesting to note that for the bicrossproduct covariant  system $( M(H), H^{*})$,  
the canonical left action  of the Hopf algebra $H^{\text{cop}}$ on $H^{*}$ is the coregular action $T_{+}$ and makes $H^{*}$ into an $H^{\text{cop}}$-module algebra by construction while the canonical left action  of  the  Hopf algebra $H$ on $H^{*}$ is the coadjoint action $L_{+}$ and makes $H^{*}$ an $H$-module coalgebra. This fits with the fact that the factor $H$ is not a sub-quantum group and we must use the correct (semidirect) coproduct of $ M(H)$ to have $H^*$ covariant. In the quantum double model the covariant action of $D(H)$ on $H$ does not define $L_{\pm}, T_{\pm}$ because they do not lead to a graph representation of $D(H)$. 
In the  $D(H)$ Kitaev lattice model, one has a $D$ module not a $D$ module algebra.



\subsection{Geometric operators}

Next, the triangle operators are used to define vertex and face operators  $A^h(v,p)$ and $B^a(v,p)$  for the bicrossproduct model on the extended  Hilbert space $\mathcal{H}_{{\Gamma}}$. These operators are also called geometric operators. Both operators depend on a pair of vertex and face that are adjacent to each other. They require linear ordering of edges at each vertex and in each face. This is specified by a site \cite{ AYK} $s=(v,p)$, which consist of a face  $p$ and adjacent vertex $v$  and represented by dotted lines as shown in Figure \ref{arbitrarygraph}. {Both vertices and faces are oriented anti-clockwise.}

In summary, the definition of $A^h(v,p)$ follows by assigning the coproduct of $h\in H$ along the edges in an anticlockwise manner taking into account the site $(v,p)$, and then the appropriate action of $h$ depending on the edge orientation. Likewise, the operator $B^a(v,p)$ is defined, but the edges associated to the face $p$ are assigned the coproduct of $a\in H^{\text{cop}}$ anticlockwise starting from the vertex $v$. The action of $a$ is then taking depending on whether the edge orientation is on the left or right of the face $p$. 

\begin{figure}[h]
\begin{center}
\includegraphics[scale=.7]{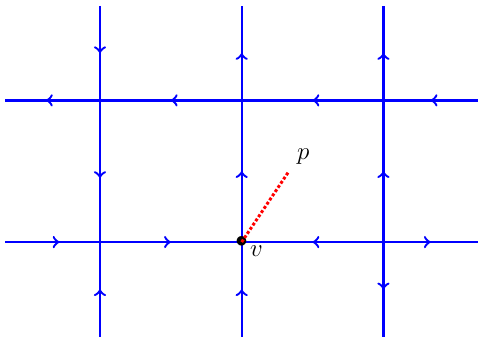}
\caption{This figure illustrates a choice of site (the dotted line) in the  graph $\Gamma$.}
\label{arbitrarygraph}
\end{center}
\end{figure}


 \newpage


 \begin{defn}\label{vex and face operators for arbitrary graph}

\textit{ Let  $(v,p)$ be the site of $\Gamma$  and $h\in H$, $a\in H^{\text{cop}}$, $\phi^i \in H^{*}$.  Let $H$ be involutive i.e., $S^2 = \id$.}

\begin{itemize} 
\item[1.]\textit{ The vertex operator $A^h(v,p):H^{*\tens |E|}\rightarrow H^{*\tens |E|}$   encodes the representation of $M(H)$ at the site  by 
\begin{equation*}
A^h(v,p)= \rho_{\pm}^{h\to}\tens\cdots\tens \rho_{\pm}^{h\tfo},
\end{equation*}
where  $\Delta^{(3)}_{M(H)}(h)=h{\to} \tens h{\ot} \tens ... \tens h{\tfo}$  with $\Delta_{M(H)}(h)$ given by the coproduct of $M(H)$ in  \eqref{coproduct of M(H)} restricted to $1\tens h$. In terms of diagrams, we have }

\begin{figure}[H]
\begin{center}
\includegraphics[scale=.8]{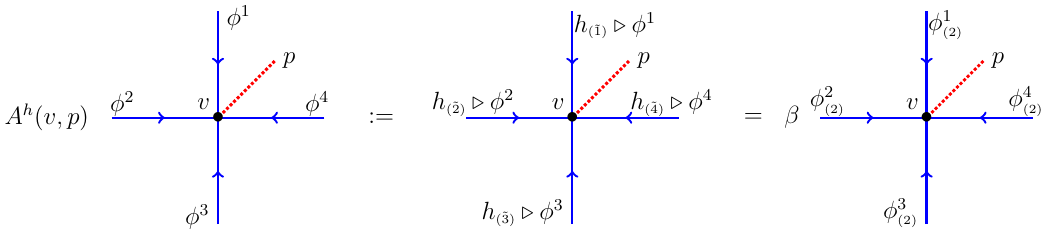}
\label{defs}
\end{center}
\end{figure}
\textit{where the  scalar $\beta\equiv\langle h,(S\phi\o^4)(S\phi\o^{3})(S\phi\o^{2})(S\phi\o^{1})\phi\thr^{1}\phi\thr^{2}\phi\thr^{3}\phi\thr^{4}\rangle$. The edges incident to the vertex $v$ are indexed counterclockwise starting from $p$. Here all the edges incident to the vertex $v$ are assumed to point towards $v$. Using the antipode to change orientation, we see that for edges orientated away from the vertex $v$, $\rho_{+}$ has to be replaced with $\rho_{-}$.}

\item[2.] \textit{  The face  operator $B^a(v,p):H^{*\tens |E|}\rightarrow H^{*\tens |E|}$ for the  face  which encodes the action of $H^{\text{cop}}$ in $M(H)$ at the site is defined by
\[
B^a(v,p)= T_{\pm}^{a\to}\tens\cdots\tens T_{\pm}^{a\tfo} 
\]
where  $\Delta^{(3)}_{M(H)}(a)=a\to\tens\cdots\tens a\tfo$ with $\Delta_{M(H)}(a)$  the coproduct of $M(H)$  \eqref{coproduct of M(H)} restricted to $a\tens 1$. More precisely, we have
}
\begin{figure}[H]
\begin{center}
\includegraphics[scale=.85]{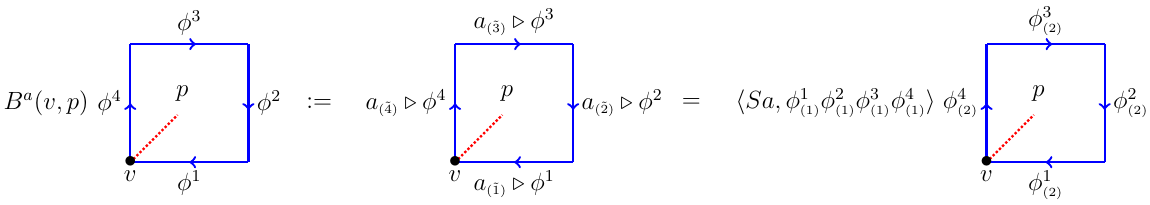}
\label{defs}
\end{center}
\end{figure}

\end{itemize}
\end{defn}


%


\bigskip

We shall now show how $\Gamma$, equipped with these operators  admits a local mirror bicrossproduct $M(H)$ representation at the sites of $\Gamma$. 
We need to show that the vertex and face operators represent their respective copies of $M(H)$ and that their commutation relations arising from common edges implement the algebra in the bicrossproduct quantum group $M(H)$.
\begin{thm}\label{thm12}
Let $H$ be a finite-dimensional Hopf algebra  satisfying $S^2 = \id$ with dual $H^*$ and the graph $\Gamma$  a square lattice as above.
 Then the operators  $A^h(v,p)$ and $ B^a(v,p) $ define an  $M(H)$ representation on $ H^{*\tens |E|}$ associated to each site $(v,p)$. Here $(a\tens h)$ acts by $A^h(v,p)\circ B^a(v,p) $, i.e. these enjoy the commutation relations 
\bes\label{AB-bicross}
 A^{h}(v,p)\circ B^{a} (v,p)&=& B^{(h\o aSh\t)}(v,p)\circ A^{h\thr}(v,p), \quad \forall h\in H,\;a\in H^{\text{cop}},\\[0.5\baselineskip]
 \label{AB-bicross1}
 A^{h}(v,p)\circ  A^{g}(v,p)&=& A^{hg}(v,p),\quad B^{a} (v,p)\circ  B^{b} (v,p) = B^{ab} (v,p).
\ees
\end{thm}

\begin{figure}[H]
\begin{center}
\includegraphics[scale=.85]{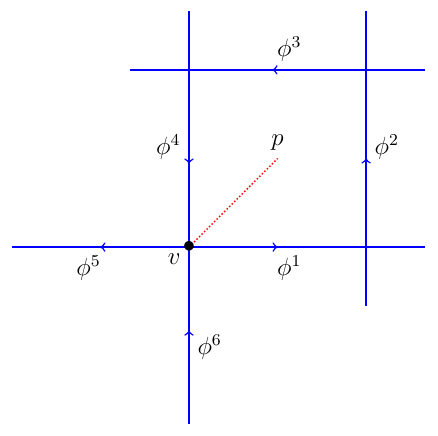}
\caption{We illustrate here a relevant subgraph of $\Gamma$, decorated by six copies of $H^{*}$ which is the relevant structure for the  proof of Theorem \ref{thm12}. Other orientations can be considered as well, but the argument follows in a similar manner. For simplicity we only focus on this graph.}
\label{arbitrary graph1}
\end{center}
\end{figure}

\newpage
\begin{proof}


Here we only show the proof of  equation \eqref{AB-bicross} and leave the proof of  \eqref{AB-bicross1} to the interested reader.
Before proceeding with the proof, for the site given in Fig. \ref{arbitrary graph1}, we note the following: 
\begin{enumerate}
\item
There are four edges connected to the vertex $v$ of Figure \ref{arbitrary graph1} and this requires we compute  $\Delta^{(3)}$ of $M(H)$ to use in the geometric operators. 
\be 
\label{1third coproduct} 
\Delta^{(3)}(a\otimes h) = a\fo \otimes h\fo\otimes a\thr h\thr Sh\fiv\otimes h\si\otimes a\t h\t Sh\sev\otimes h\ei\tens a\o h\o Sh\Nine\tens h\ten
\ee
from which we derive that
\bes
\label{1thirdcoproduct} 
\Delta^{(3)}(a\otimes 1) &=& a\fo \otimes 1\otimes a\thr \otimes 1\otimes a\t \otimes 1\tens a\o \tens  1\nonumber\\
\Delta^{(3)}(1\otimes h) &=& 1 \otimes h\fo\otimes  h\thr Sh\fiv\otimes h\si\otimes  h\t Sh\sev\otimes h\ei\tens  h\o Sh\Nine\tens h\ten.
\ees
These are the relevant expressions to use in the definition of the face and vertex operators.  
\item The face operator associated with Figure \ref{arbitrary graph1} is  defined as
\be
B^{a}(v,p) ( \phi^{1}\otimes...\otimes \phi^{6}) =  T^{a\fo}_{-} (\phi^{1})\tens T^{a\thr}_{-}(\phi^{2})\tens T^{a
\t}_{-}(\phi^{3})\tens T^{a\o}_{-}(\phi^{4})\tens\phi^{5}\tens \phi^{6}
\ee
and the vertex operator is 
\be 
\label{verx}
A^{h}(v,p)( \phi^{1}\otimes...\otimes \phi^{6}) = \rho^{h\to}_{-}(\phi^1)\tens \phi^2\tens \phi^3 \tens \rho^{h\ot}_{+}(\phi^4)\tens \rho^{h\tthr}_{-}(\phi^5)\tens \rho^{h\tfo}_{+}(\phi^6)
\ee
where $\rho_{\pm}$ are  expressed in terms of the triangle operators as follows:
\bes
\rho^{h\to}_{-}& \equiv& \rho^{1\tens h\fo}_{-} = L^{h\fo}_{-}, \qquad \qquad \rho^{h\ot}_{+}\equiv \rho^{h\thr Sh\fiv\otimes h\si}_{+}= T^{h\thr Sh\fiv}_{+}L^{h\si}_{+} \nonumber\\[0.5\baselineskip]
\rho^{h\tthr}_{-} &\equiv & \rho^{h\t Sh\sev\otimes h\ei}_{-}= T^{h\t Sh\sev}_{-}L^{h\ei}_{-},\qquad \rho^{\tfo}_{+}\equiv \rho^{h\o Sh\Nine\tens h\ten}_{+}= T^{h\o Sh\Nine}_{+}L^{h\ten}_{+}.\nonumber
\ees
Note that $\rho_{-}$ are computed from \eqref{rhominus}.
The vertex operator \eqref{verx} then reads
\bes
A^{h}(v,p)( \phi^{1}\otimes...\otimes \phi^{6})&=& L^{ h\fo}_{-}(\phi^{1})\otimes \phi^2\tens\phi^{3} \tens T^{h\thr Sh\fiv}_{+}L^{h\si}_{+}(\phi^{4}) \nonumber\\[0.5\baselineskip]
&\quad &\,\otimes\, T^{h\t Sh\sev}_{-}L^{h\ei}_{-}(\phi^{5}) \otimes\, T_{+}^{h\o Sh\Nine}L^{h\ten}_{+}(\phi^{6}).
\ees
\end{enumerate}

\newpage
The proof follows a direct  calculation to compare the LHS and RHS of \eqref{AB-bicross}. 
%
Let $h\in H$, $a\in H^{\text{cop}}$ and $\phi^{i}\in H^{*}$, where $i\in \lbrace 1,..,6\rbrace$. Starting with the LHS of equation \eqref{AB-bicross}, we have
\bes
\label{1BC5}
&&A^{h}(v,p)B^{a}(v,p)( \phi^{1}\otimes...\otimes \phi^{6})
= A^{h}(v,p)\left(T^{a\fo}_{-} (\phi^{1})\tens T^{a\thr}_{-}(\phi^{2})\tens T^{a
\t}_{-}(\phi^{3})\tens T^{a\o}_{-}(\phi^{4})\tens\phi^{5}\tens \phi^{6}\right) \nonumber\\[0.3\baselineskip]
&=& L^{h\fo}_{-}T^{a\fo}_{-} (\phi^{1})\tens  T^{a\thr}_{-}(\phi^{2})\tens T^{a
\t}_{-}(\phi^{3})\tens T^{h\thr Sh\fiv}_{+}L^{h\si}_{+}\left( T^{a\o}_{-}(\phi^{4})\right) \nonumber\\[0.3\baselineskip]
&& \otimes\, T^{h\t Sh\sev}_{-}L^{h\ei}_{-}(\phi^{5}) \otimes\, T_{+}^{h\o Sh\Nine}L^{h\ten}_{+}(\phi^{6}) \nonumber\\[0.3\baselineskip]
&=& \langle a\fo,\phi\fo^{1}\rangle\langle h\fo,\phi\thr^{1}S^{-1}\phi\o^{1}\rangle\phi\t^{1}\tens\, \langle a\thr,\phi\t^{2}\rangle \phi\o^{2}\tens\, \langle a\t,\phi\t^{3}\rangle\phi\o^{3}\nonumber\\[0.3\baselineskip]
&& \tens\,\langle a\o,\phi\si^{4}\rangle\langle h\si, S\phi\o^{4}\phi\fiv^{4}\rangle\langle h\fiv,\phi\t^{4}\rangle\langle Sh\thr, \phi\thr^{4}\rangle\phi\fo^{4} \nonumber\\[0.3\baselineskip]
&&\tens\,\langle h\ei,\phi\fiv^{5}S^{-1}\phi\o^{5}\rangle\langle h\t,\phi\thr^{5}\rangle\langle Sh\sev,\phi\fo^{5}\rangle\phi\t^{5} \tens\,\langle h\ten,S\phi\o^{6}\phi\fiv^{6}\rangle\langle h\Nine, \phi\t^{6}\rangle\langle Sh\o,\phi\thr^{6}\rangle\phi\fo^{6} \nonumber\\[0.3\baselineskip]
&=& \langle a\fo,\phi_{\t\t}^{1}\rangle\langle h\fo,\phi_{\t\o}^{1}\rangle\langle S^{-1}h\fiv, \phi_{\o\o}^{1}\rangle\phi_{\o\t}^{1}\tens\, \langle a\thr,\phi\t^{2}\rangle \phi\o^{2}\tens\, \langle a\t,\phi\t^{3}\rangle\phi\o^{3} \nonumber\\[0.3\baselineskip]
&& \tens\,\langle a\o,\phi\si^{4}\rangle\langle h\si,\phi\t^{4}\rangle\langle h\sev, S\phi\o^{4}\phi\fiv^{4}\rangle\langle Sh\thr, \phi\thr^{4}\rangle\phi\fo^{4}\nonumber\\[0.3\baselineskip]
&& \tens\,\langle h\Nine,\phi\fiv^{5}S^{-1}\phi\o^{5}\rangle\langle h\t,\phi\thr^{5}\rangle\langle h\ei,S\phi\fo^{5}\rangle\phi\t^{5} \tens\,\langle h\ele,S\phi\o^{6}\phi\fiv^{6}\rangle\langle h\ten, \phi\t^{6}\rangle\langle Sh\o,\phi\thr^{6}\rangle\phi\fo^{6} \nonumber\\[0.3\baselineskip]
&=& \langle a\fo,\phi\thr^{1}\rangle\langle S^{-1}h\fiv h\fo, \phi\o^{1}\rangle\phi\t^{1}\tens\, \langle a\thr,\phi\t^{2}\rangle \phi\o^{2}\tens\, \langle a\t,\phi\t^{3}\rangle\phi\o^{3}  \nonumber\\[0.3\baselineskip]
&& \tens\,\langle a\o,\phi\si^{4}\rangle\langle h\si,\phi\t^{4} S\phi\o^{4}\phi\fiv^{4}\rangle\langle Sh\thr, \phi\thr^{4}\rangle\phi\fo^{4}\nonumber\\[0.3\baselineskip]
&& \tens\,\langle h\t,\phi\thr^{5}\rangle\langle h\sev,S\phi\fo^{5}\phi\fiv^{5}S^{-1}\phi\o^{5}\rangle\phi\t^{5}\nonumber  \tens\langle h\ei, \phi\t^{6}S\phi\o^{6}\phi\fiv^{6}\rangle\langle Sh\o,\phi\thr^{6}\rangle\phi\fo^{6}\nonumber\\[0.3\baselineskip]
&=& \langle a\fo,\phi\t^{1}\rangle\phi\o^{1} \epsilon(h\fo)\tens\, \langle a\thr,\phi\t^{2}\rangle \phi\o^{2}\tens\, \langle a\t,\phi\t^{3}\rangle\phi\o^{3} \tens\,\langle a\o,\phi\fo^{4}\rangle\langle h\fiv,\phi\thr^{4}\rangle\langle Sh\thr, \phi\o^{4}\rangle\phi\t^{4}\nonumber\\[0.3\baselineskip]
&& \tens\,\langle h\t,\phi\thr^{5}\rangle\langle h\si,S^{-1}\phi\o^{5}\rangle\phi\t^{5}\nonumber  \tens\,\langle h\sev, \phi\thr^{6}\rangle\langle Sh\o,\phi\o^{6}\rangle\phi\t^{6}\nonumber\\[0.3\baselineskip]
&=& \langle a\fo,\phi\t^{1}\rangle\phi\o^{1} \tens\, \langle a\thr,\phi\t^{2}\rangle \phi\o^{2}\tens\, \langle a\t,\phi\t^{3}\rangle\phi\o^{3} \tens\,\langle a\o,\phi_{\t\t}^{4}\rangle\langle h\fo,\phi_{\t\o}^{4}\rangle\langle Sh\thr, \phi_{\o\o}^{4}\rangle\phi_{\o\t}^{4}\nonumber\\[0.3\baselineskip]
&& \tens\,\langle h\t,\phi\thr^{5}\rangle\langle h\fiv,S^{-1}\phi\o^{5}\rangle\phi\t^{5}\nonumber  \tens\,\langle h\si, \phi\thr^{6}\rangle\langle Sh\o,\phi\o^{6}\rangle\phi\t^{6}\nonumber\\[0.3\baselineskip]
&=& \langle a\fo,\phi\t^{1}\rangle\phi\o^{1} \tens\, \langle a\thr,\phi\t^{2}\rangle \phi\o^{2}\tens\, \langle a\t,\phi\t^{3}\rangle\phi\o^{3} \tens\,\langle a\o,\phi\t^{4}\rangle\phi\o^{4}\epsilon(h\thr)\nonumber\\[0.3\baselineskip]
&& \tens\,\langle h\t,\phi\thr^{5}\rangle\langle h\fo,S^{-1}\phi\o^{5}\rangle\phi\t^{5}\nonumber  \tens\,\langle h\fiv, \phi\thr^{6}\rangle\langle Sh\o,\phi\o^{6}\rangle\phi\t^{6}\nonumber\\[0.3\baselineskip]
&=& \langle a,\phi\t^{1}\phi\t^{2}\phi\t^{3}\phi\t^{4}\rangle \langle h\t,\phi\thr^{5}\rangle\langle h\thr,S^{-1}\phi\o^{5}\rangle\langle Sh\o,\phi\o^{6}\rangle\langle h\fo,\phi\thr^{6}\rangle \nonumber\\[0.3\baselineskip]
&& \,\phi\o^{1}\otimes \phi\o^2\tens\phi\o^{3} \tens \phi\o^{4} \tens \phi\t^{5}\tens\phi\t^{6}.
\ees
Using the first part of \eqref{1thirdcoproduct}, we calculate
\be 
\Delta^3\left(h\o aSh\t\right)= h\fo a\fo Sh\fiv \tens h\thr a\thr Sh\si \tens h\t a\t Sh\sev \tens h\o a\o Sh\ei 
\ee 
and this will be used in computing the RHS of  \eqref{AB-bicross}. 
The RHS of equation \eqref{AB-bicross}, is computed as follows

\bes
&& B^{h\o aSh\t}(v,p)A^{h\thr}(v,p)(\phi^{1}\otimes...\otimes \phi^{6})\nonumber\\
&=& B^{h\o aSh\t}(v,p)\left(L^{h\si}_{-}(\phi^{1})\otimes \phi^2\tens\phi^{3} \tens T^{h\fiv Sh\sev}_{+}L^{h\ei}_{+}(\phi^{4})\,\otimes T^{h\fo Sh\Nine}_{-}L^{h\ten}_{-}(\phi^{5}) \otimes\, T_{+}^{h\thr Sh\ele}L^{h\twe}_{+}(\phi^{6})\right)\nonumber\\[0.3\baselineskip]
&=& T^{h\fo a\fo Sh\fiv}_{-} L^{h\twe}_{-}(\phi^{1})\tens T^{ h\thr a\thr Sh\si}_{-}(\phi^{2})\tens T^{ h\t a\t Sh\sev}_{-}(\phi^{3})\tens T^{h\o a\o Sh\ei}_{-} T^{h\ele Sh\thi}_{+}L^{h\fou}_{+}(\phi^{4})\nonumber\\[0.3\baselineskip]
&& \tens\, T^{h\ten Sh\fif}_{-}L^{h\sixt}_{-}(\phi^{5})\tens T^{h\Nine Sh\sevt}_{+}L^{h\eit}_{+}(\phi^5)\nonumber\\[0.3\baselineskip]
&=& \langle h\twe, \phi\fo^{1}S^{-1}\phi\o^{1}\rangle\langle h\fo a\fo Sh\fiv,\phi\thr^{1} \rangle\phi\t^{1}\,\tens\, \langle h\thr,\phi\t^{2}\rangle\langle a\thr,\phi\thr^{2}\rangle\langle Sh\si,\phi\fo^{2}\rangle\phi\o^{2}\nonumber\\[0.3\baselineskip]
&& \tens\,\langle h\t, \phi\t^{3}\rangle\langle a\t,\phi\thr^3\rangle \langle Sh\sev,\phi\fo^{3}\rangle\phi\o^{3}\,\tens\,  \langle h\fou, S\phi\o^{4}\phi\fiv^{4}\rangle \langle h\thi Sh\ele, \phi\t^{4} \rangle\langle h\o a\o Sh\ei, \phi\fo^{4}\rangle\phi\thr^{4}\nonumber\\[0.3\baselineskip]
&& \tens\,\langle h\sixt,\phi\fiv^{5}S^{-1}\phi\o^{5}\rangle\langle h\ten, \phi\thr^{5}\rangle\langle Sh\fif,\phi\fo^{5}\rangle\phi\t^{5}\,\tens\, \langle h\eit, S\phi\o^{6}\phi\fiv^{6}\rangle\langle h\sevt,\phi\t^{6}\rangle\langle Sh\Nine, \phi\thr^{6}\rangle\phi\fo^{6}
\nonumber\\[0.3\baselineskip]
&=&  \langle h\twe, \phi\fo^{1}S^{-1}\phi\o^{1}\rangle\langle h\fo a\fo Sh\fiv,\phi\thr^{1} \rangle \phi\t^{1}\,\tens\,  \langle h\thr,\phi\t^{2}\rangle\langle a\thr,\phi\thr^{2}\rangle\langle Sh\si,\phi\fo^{2}\rangle\phi\o^{2}\nonumber\\[0.3\baselineskip]
&& \tens\,\langle h\t, \phi\t^{3}\rangle\langle a\t,\phi\thr^3\rangle \langle Sh\sev,\phi\fo^{3}\rangle\phi\o^{3}\,\tens\,  \langle h\fou, S\phi\o^{4}\phi\fiv^{4}\rangle \langle h\thi Sh\ele, \phi\t^{4} \rangle\langle h\o a\o Sh\ei, \phi\fo^{4}\rangle\phi\thr^{4}\nonumber\\[0.3\baselineskip]
&& \tens\,\langle h\fif,S\phi\fo^{5}\rangle\langle h\sixt,\phi\fiv^{5}S^{-1}\phi\o^{5}\rangle\langle h\ten, \phi\thr^{5}\rangle\phi\t^{5}\,\tens\,\langle h\sevt,\phi\t^{6}\rangle \langle h\eit, S\phi\o^{6}\phi\fiv^{6}\rangle\langle Sh\Nine, \phi\thr^{6}\rangle\phi\fo^{6}
\nonumber\\[0.3\baselineskip]
&=&  \langle h\twe, (\phi\t^{1})\t(S^{-1}\phi\o^{1})\t\rangle\langle h\fo a\fo Sh\fiv,(\phi\t^{1})\o \rangle S(S^{-1}\phi\o^{1})\o\,\tens\,  \langle h\thr,\phi\t^{2}\rangle\langle a\thr,\phi\thr^{2}\rangle\langle Sh\si,\phi\fo^{2}\rangle\phi\o^{2}\nonumber\\[0.3\baselineskip]
&& \tens\,\langle h\t, \phi\t^{3}\rangle\langle a\t,\phi\thr^3\rangle \langle Sh\sev,\phi\fo^{3}\rangle\phi\o^{3}\,\tens\,  \langle h\thi,\phi\fo^{4}\rangle \langle  Sh\ele, \phi\o^{4} \rangle\langle h\o a\o Sh\ei, \phi\thr^{4}\rangle\phi\t^{4}\nonumber\\[0.3\baselineskip]
&& \tens\,\langle h\fou,S^{-1}\phi\o^{5}\rangle\langle h\ten, \phi\thr^{5}\rangle\phi\t^{5}\,\tens\,\langle h\fif,\phi\thr^{6}\rangle \langle Sh\Nine, \phi\o^{6}\rangle\phi\t^{6}\nonumber\\[0.3\baselineskip]
&=&  \langle h\twe, (\phi\t^{1})\t(S^{-1}\phi\o^{1})\t\rangle\langle h\fo a\fo Sh\fiv,(\phi\t^{1})\o \rangle S(S^{-1}\phi\o^{1})\o\,\tens\,  \langle h\thr,\phi\t^{2}\rangle\langle a\thr,\phi\thr^{2}\rangle\langle Sh\si,\phi\fo^{2}\rangle\phi\o^{2}\nonumber\\[0.3\baselineskip]
&& \tens\,\langle h\t, \phi\t^{3}\rangle\langle a\t,\phi\thr^3\rangle \langle Sh\sev,\phi\fo^{3}\rangle\phi\o^{3}\,\tens\,  \langle h\thi,\phi\fo^{4}\rangle \langle  Sh\ele, \phi\o^{4} \rangle\langle h\o a\o Sh\ei, \phi\thr^{4}\rangle\phi\t^{4}\nonumber\\[0.3\baselineskip]
&& \tens\,\langle h\fou,S^{-1}\phi\o^{5}\rangle\langle h\ten, \phi\thr^{5}\rangle\phi\t^{5}\,\tens\,\langle h\fif,\phi\thr^{6}\rangle \langle Sh\Nine, \phi\o^{6}\rangle\phi\t^{6}\nonumber\\[0.3\baselineskip]
&=& \langle h\fo a\fo Sh\fiv,\phi\t^{1} \rangle \phi\o^{1}\,\tens\,  \langle h\thr,\phi\t^{2}\rangle\langle a\thr,\phi\thr^{2}\rangle\langle Sh\si,\phi\fo^{2}\rangle\phi\o^{2}\nonumber\\[0.3\baselineskip]
&& \tens\,\langle h\t, \phi\t^{3}\rangle\langle a\t,\phi\thr^3\rangle \langle Sh\sev,\phi\fo^{3}\rangle\phi\o^{3}\,\tens\,  \langle h\twe,\phi\fo^{4}\rangle \langle  Sh\ele, \phi\o^{4} \rangle\langle h\o a\o Sh\ei, \phi\thr^{4}\rangle\phi\t^{4}\nonumber\\[0.3\baselineskip]
&& \tens\,\langle h\thi,S^{-1}\phi\o^{5}\rangle\langle h\ten, \phi\thr^{5}\rangle\phi\t^{5}\,\tens\,\langle h\fou,\phi\thr^{6}\rangle \langle Sh\Nine, \phi\o^{6}\rangle\phi\t^{6}
\nonumber\\[0.3\baselineskip]
&=& \langle h\fo, \phi_{\o\t}^{1}\rangle\langle a\fo,\phi_{\t\o}^{1}\rangle\langle Sh\fiv,\phi_{\t\t}^{1}\rangle \phi_{\o\o}^{1}\,\tens\,  \langle h\thr,\phi\t^{2}\rangle\langle a\thr,\phi\thr^{2}\rangle\langle Sh\si,\phi\fo^{2}\rangle\phi\o^{2}\nonumber\\[0.3\baselineskip]
&& \tens\,\langle h\t, \phi\t^{3}\rangle\langle a\t,\phi\thr^3\rangle \langle Sh\sev,\phi\fo^{3}\rangle\phi\o^{3}\,\tens\,  \langle h\twe,\phi\fo^{4}\rangle \langle  Sh\ele, \phi\o^{4} \rangle\langle h\o a\o Sh\ei, \phi\thr^{4}\rangle\phi\t^{4}\nonumber\\[0.3\baselineskip]
&& \tens\,\langle h\thi,S^{-1}\phi\o^{5}\rangle\langle h\ten, \phi\thr^{5}\rangle\phi\t^{5}\,\tens\,\langle h\fou,\phi\thr^{6}\rangle \langle Sh\Nine, \phi\o^{6}\rangle\phi\t^{6}\nonumber\\[0.3\baselineskip]
&=& \langle a\fo,\phi\t^{1}\rangle\phi\o^{1} \,\tens\,  \langle h\thr,\phi_{\o\t}^{2}\rangle\langle a\thr,\phi_{\t\o}^{2}\rangle\langle Sh\fo,\phi_{\t\t}^{2}\rangle\phi_{\o\o}^{2}\nonumber\\[0.3\baselineskip]
&& \tens\,\langle h\t, \phi\t^{3}\rangle\langle a\t,\phi\thr^3\rangle \langle Sh\fiv,\phi\fo^{3}\rangle\phi\o^{3}\,\tens\,  \langle h\ten,\phi\fo^{4}\rangle \langle  Sh\Nine, \phi\o^{4} \rangle\langle h\o a\o Sh\si, \phi\thr^{4}\rangle\phi\t^{4}\nonumber\\[0.3\baselineskip]
&& \tens\,\langle h\ele,S^{-1}\phi\o^{5}\rangle\langle h\ei, \phi\thr^{5}\rangle\phi\t^{5}\,\tens\,\langle h\twe,\phi\thr^{6}\rangle \langle Sh\sev, \phi\o^{6}\rangle\phi\t^{6}
\nonumber\\[0.3\baselineskip]
&=& \langle a\fo,\phi\t^{1}\rangle\phi\o^{1} \,\tens\, \langle a\thr,\phi\t^{2}\rangle\phi\o^{2} \tens\,\langle a\t,\phi_{\t\o}^3\rangle\langle h\t, \phi_{\o\t}^{3}\rangle \langle Sh\thr,\phi_{\t\t}^{3}\rangle\phi_{\o\o}^{3}
\nonumber\\[0.3\baselineskip]
&& \tens\,  \langle h\ei,\phi\si^{4}\rangle \langle  Sh\sev, \phi\o^{4} \rangle\langle h\o, \phi\thr^{4}\rangle\langle a\o,\phi\fo^{4}\rangle\langle Sh\fo,\phi\fiv^{4}\rangle\phi\t^{4}
\nonumber\\[0.3\baselineskip]
&& \tens\,\langle h\Nine,S^{-1}\phi\o^{5}\rangle\langle h\si, \phi\thr^{5}\rangle\phi\t^{5}\,\tens\,\langle h\ten,\phi\thr^{6}\rangle \langle Sh\fiv, \phi\o^{6}\rangle\phi\t^{6}.
\ees

It follows that
\bes
\label{1RHSthm}
&& B^{h\o aSh\t}(v,p)A^{h_{\thr}}(v,p)(\phi^{1}\otimes...\otimes \phi^{6})= \langle a\fo,\phi\t^{1}\rangle\phi\o^{1} \,\tens\, \langle a\thr,\phi\t^{2}\rangle\phi\o^{2} \tens\,\langle a\thr,\phi\t^{3}\rangle\phi\o^{3} 
\nonumber\\[0.3\baselineskip]
&&\tens\, \langle a\o,\phi_{\t\t}^{4}\rangle \langle h\si,\phi_{\thr\t}^{4}\rangle \langle  Sh\fiv, \phi_{\o\o}^{4} \rangle\langle h\o, \phi_{\t\o}^{4}\rangle\langle Sh\t,\phi_{\thr\o}^{4}\rangle\phi_{\o\t}^{4}\nonumber\\[0.3\baselineskip]
&& \tens\,\langle h\sev,S^{-1}\phi\o^{5}\rangle\langle h\fo, \phi\thr^{5}\rangle\phi\t^{5}\,\tens\,\langle h\ei,\phi\thr^{6}\rangle \langle Sh\thr, \phi\o^{6}\rangle\phi\t^{6}
\nonumber\\[0.3\baselineskip]
&=& \langle a\fo,\phi\t^{1}\rangle\phi\o^{1} \,\tens\, \langle a\thr,\phi\t^{2}\rangle\phi\o^{2} \tens\,\langle a\thr,\phi\t^{3}\rangle\phi\o^{3}\, \tens\, \langle a\o,\phi\thr^{3}\rangle \langle h\fo,\phi\fo^{4}\rangle \langle  Sh\thr, \phi\o^{4} \rangle\phi\t^{4}\nonumber\\[0.3\baselineskip]
&& \tens\,\langle h\fiv,S^{-1}\phi\o^{5}\rangle\langle h\t, \phi\thr^{5}\rangle\phi\t^{5}\,\tens\,\langle h\si,\phi\thr^{6}\rangle \langle Sh\o, \phi\o^{6}\rangle\phi\t^{6}
\nonumber\\[0.3\baselineskip]
&=& \langle a\fo,\phi\t^{1}\rangle\phi\o^{1} \,\tens\, \langle a\thr,\phi\t^{2}\rangle\phi\o^{2} \tens\,\langle a\thr,\phi\t^{3}\rangle\phi\o^{3}\, \tens\, \langle a\o,\phi_{\t\o}^{3}\rangle \langle  h\thr, (S\phi\o^{4})\t \phi_{\t\t}^{4}\rangle S^{-1}(S\phi\o^4)\o\nonumber\\[0.3\baselineskip]
&& \tens\,\langle h\fo,S^{-1}\phi\o^{5}\rangle\langle h\t, \phi\thr^{5}\rangle\phi\t^{5}\,\tens\,\langle h\fiv,\phi\thr^{6}\rangle \langle Sh\o, \phi\o^{6}\rangle\phi\t^{6}\nonumber\\[0.3\baselineskip]
&=& \langle a\fo,\phi\t^{1}\rangle\phi\o^{1} \,\tens\, \langle a\thr,\phi\t^{2}\rangle\phi\o^{2} \tens\,\langle a\thr,\phi\t^{3}\rangle\phi\o^{3}\, \tens\, \langle a\o,\phi\t^{3}\rangle \phi\o^{4}\nonumber\\[0.3\baselineskip]
&& \tens\,\langle h\thr,S^{-1}\phi\o^{5}\rangle\langle h\t, \phi\thr^{5}\rangle\phi\t^{5}\,\tens\,\langle h\fo,\phi\thr^{6}\rangle \langle Sh\o, \phi\o^{6}\rangle\phi\t^{6}\nonumber\\[0.3\baselineskip]
&=& \langle a,\phi\t^{1}\phi\t^{2}\phi\t^{3}\phi\t^{4}\rangle \langle h\t,\phi\thr^{5}\rangle\langle h\thr,S^{-1}\phi\o^{5}\rangle\langle Sh\o,\phi\o^{6}\rangle\langle h\fo,\phi\thr^{6}\rangle \nonumber\\[0.3\baselineskip]
&& \,\phi\o^{1}\otimes \phi\o^2\tens\phi\o^{3} \tens \phi\o^{4} \tens \phi\t^{5}\tens\phi\t^{6}.
\ees
Finally we see equations \eqref{1BC5} and \eqref{1RHSthm} are the same, and hence this proves Theorem \ref{thm12}  
for the specific graph in \ref{arbitrary graph1}.  The variations according to different orientations proceed similarly and are omitted.

\end{proof}

When the geometric operators do not act at the same site, they essentially commute as stated in the following proposition.
\begin{pro}
\label{ABproperties}\textit{ 
Let  $\Gamma$ be the square lattice described above, and $h,g\in H$ and $a,b\in H^{\text{cop}}$.
\begin{itemize}
\item[(i)] For all  sites,  $A^{h}(v,p) \circ A^{g}(w,q) =A^{g}(w,q) \circ A^{h}(v,p)$ provided the vertices $v$ and $w$ do not coincide, i.e., $v\neq w$ and there are no edges connecting to $v$ and $w$ directly.
\item[(ii)] For all sites, $B^{a}(v,p)\circ B^{b}(w,q)= B^{b}(w,q)\circ B^{a}(v,p)$ if the faces $p$ and $q$ do not coincide, i.e., $p\neq q$.
\item[(iii)] At disjoint sites, {$v\neq v'$ and $p\neq p'$}, $A^{h}(v,p)\circ B^{b}(v',p')=B^{b}(v',p')\circ A^{h}(v,p)$.
\end{itemize}
}
\end{pro}
\begin{proof} 
 Before we prove the above proposition, we  note
 triangle operators \eqref{triangle operators} obey the following commutation relations 
\begin{eqnarray}
\label{triangle operators algebra}  
[T_{+}^{a},T_{-}^{b}]&=&0,\quad   L^{h}_{\pm} T^{a}_{\pm} = T^{h\o aSh\t}_{\pm} L^{h\thr}_{\pm}, \, \text{for all}\;\; h,g\in H,\;a,b\in H^{\text{cop}}.
\end{eqnarray}

(i) Consider the two graphs in Figure \ref{Lemma1} with vertices $v$ and $w$ which do not coincide and have no edges connecting them.   Then it can be shown easily that  the graph (with $p\neq q$) on the left has the following  vertex operators $A^{h}(v,p)$ and $A^{g}(w,q)$ commuting. If one considers the  graph (with $p=q$) on the right side of  Figure \ref{Lemma1}, the vertex operators are computed as 
\bes
&&A^{h}(v,p)(\phi^{1} \tens \phi^{2}\tens \phi^{3}\tens \phi^{4}\tens\psi^1\tens\psi^2\tens\psi^3\tens\psi^4)\nonumber\\[0.3\baselineskip]
&=& \langle h,(S\phi\o^4)(S\phi\o^{3})(S\phi\o^{2})(S\phi\o^{1})\phi\thr^{1}\phi\thr^{2}\phi\thr^{3}\phi\thr^{4}\rangle
\phi\t^{1} \tens \phi\t^{2}\tens \phi\t^{3}\tens \phi\t^{4}\tens\psi^1\tens\psi^2\tens\psi^3\tens\psi^4 \nonumber\\[0.3\baselineskip]
&&A^{g}(w,p)(\phi^{1} \tens \phi^{2}\tens \phi^{3}\tens \phi^{4}\tens\psi^1\tens\psi^2\tens\psi^3\tens\psi^4)\nonumber\\[0.3\baselineskip]
&=& \langle g,(S\psi\o^4)(S\psi\o^{3})(S\psi\o^{2})(S\psi\o^{1})\psi\thr^{1}\psi\thr^{2}\psi\thr^{3}\psi\thr^{4}\rangle
\phi^{1} \tens \phi^{2}\tens \phi^{3}\tens \phi^{4}\tens\psi\t^1\tens\psi\t^2\tens\psi\t^3\tens\psi\t^4. \nonumber 
\ees
A straightforward calculation shows that the above computed vertex operators commute, {essentially because the operators do not act on the same Hilbert spaces. The choice of orientation of the edges is not changing this.  A similar calculation applies for the three other cases where $v$ and $w$ are diagonally opposite on the same square, and $p$ and $q$ are different. }

%
\begin{figure}[H]
\begin{center}
\includegraphics[scale=.75]{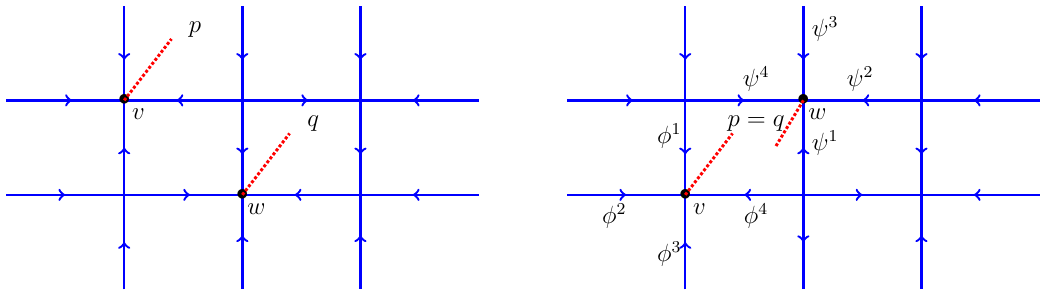}
\caption{Diagrams showing two different graphs each with vertices $v$ and $w$ which do not coincide. {We could also consider three more cases, where $v$ and $w$ are diagonally opposite and $p$ and $q$ different. }}
\label{Lemma1}
\end{center}
\end{figure}


\medskip

(ii) Given the graph on the right of Figure \ref{Lemmaiii} with  no edge(s) connecting the faces $p$ and $q$, it is easy to see that their respective  face operators $B^{a}(v,p)$ and $B^{b}(w,q)$ commute. {Indeed, in this case, the operators act on non-overlapping Hilbert spaces (regardless of the orientation).}

Let us then assume that the face $p$ and the face $q$ share at least a common edge (with dashed line) as shown by the graph on the left of Figure \ref{Lemmaiii}. Then  we only need to verify that the operators $B^{a}(v,p)$ and $B^{b}(w,q)$ commute on their shared edge. This common edge is oriented such that it is on the left of $p$ and on the right of $q$. Then $B^{a}(v,p)$ will acts on the edge through $T_{-}$ whiles $B^{b}(w,q)$ will acts on the edge through $T_{+}$. From  \eqref{triangle operators algebra} we know that $T_{-}$ and $T_{+}$ commute implying that  the operators $B^{a}(v,p)$ and $B^{b}(w,q)$ commute also. {Changing the orientation of the edges is not affecting this result.}

\medskip
(iii) Consider Figure \ref{Lemmaiii} below with two different graphs each with disjoint sites $(v,p)$ and $(v',p')$. Here by disjoint sites we mean $(v,p) \neq(v',p')$ though common edges shared by faces $p$ and $q$ are allowed leading to $p$ and $q$  adjacent to each other  as illustrated by the graph on the left of Figure \ref{Lemmaiii}.  The second graph on the right of Figure \ref{Lemmaiii}  however has no edges common to both sites.

For both graphs of  Figure \ref{Lemmaiii} their respective vertex operator for the site $(v,p)$ moving anticlockwise  are the same (changing the orientation of either graph does not change the final outcome) 
\begin{figure}[H]
\begin{center}
\includegraphics[scale=.65]{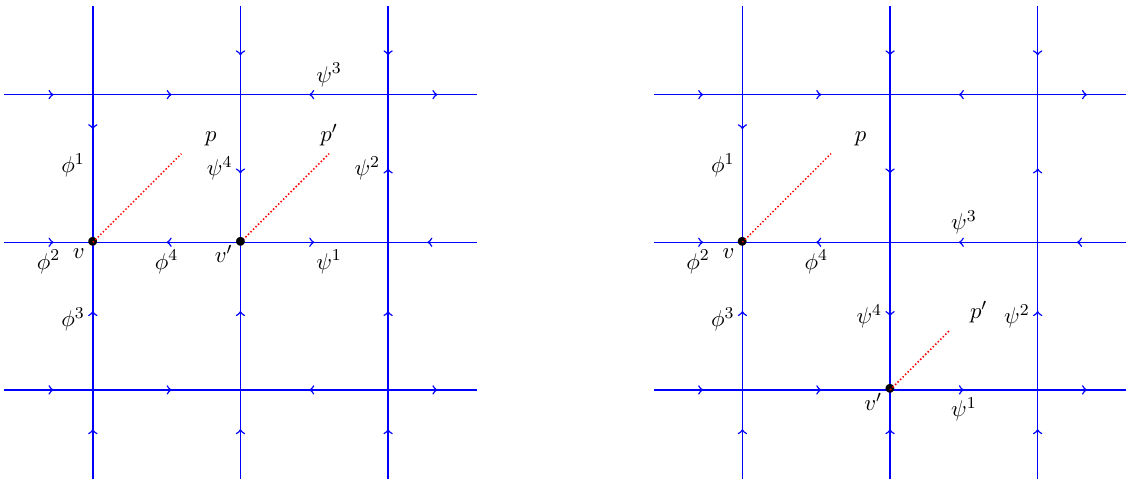}
\caption{ Diagram depicting two different graphs each with disjoint sites $(v,p)$ and $(v',p')$. }
\label{Lemmaiii}
\end{center}
\end{figure}
%
\bes 
&&A^{h}(v,p)(\phi^{1} \tens \phi^{2}\tens \phi^{3}\tens \phi^{4}\tens\psi^1\tens\psi^2\tens\psi^3\tens\psi^4)\\[0.3\baselineskip]
 &=& L^{ h\fo}_{+}(\phi^{1})\tens T^{h\thr Sh\fiv}_{+} L^{h\si}_{+}( \phi^{2})\tens T^{h\t Sh\sev}_{+}L^{h\ei}_{+}(\phi^{3})\tens T^{h\o Sh\Nine}_{+} L^{h\ten}_{+}(\phi^{4})\tens\psi^1\tens\psi^2\tens\psi^3\tens\psi^4\nonumber
\ees
and likewise the  face operator for both graphs is the same for site  $(v',p')$ moving counterclockwise   
\bes
&& B^{b}(v',p')(\phi^{1}\otimes \phi^{2}\tens\phi^{3}\tens\phi^{4}\tens\psi^1\tens\psi^2\tens\psi^3\tens\psi^4)\nonumber\\[0.3\baselineskip]
 &=& \phi^1\tens\phi^2\tens\phi^3\tens\phi^4\tens T^{b\fo}_{-}(\psi^{1})\tens  T^{b\thr}_{-}(\psi^{2})\tens T^{b\t}_{-}(\psi^{3})\tens T^{b\o}_{-}(\psi^{4}).
\ees
{These operators commute as they act on non-overlapping Hilbert spaces. Changing the orientation of the edges would not affect this. }
\end{proof}

\subsection{Hilbert space and Hamiltonian }

We are almost ready to define the Hamiltonian of the semidual  Kitaev  model. The last piece we need to define is an inner product on $H_\Gamma = H^{\tens|E|}$ making it into a Hilbert space, and to define a self-adjoint Hamiltonian on it. To this end, we henceforth work over $k= \C$ and, moreover, we suppose that $H$ is a finite- dimensional Hopf $\C^\star$-algebra e.g. as in \cite{KLS,NS}.
 
To define such Hilbert space, we require $H$ be a finite-dimensional Hopf $C^{\star}$-algebra.  This choice of $H$ permits for the definition of an inner product and a $\star$-representation for the model.    This kind of Hopf algebra comes equipped with normalized Haar integrals, structures which are natural candidates to define an inner product.  We recall that a normalised Haar integral in any Hopf algebra $H$ means $\ell\in H$ such that
\be h\cdot\ell =\ell\cdot h=\epsilon(h)\ell, \quad \text{for all } h\in H \text{ and } \epsilon(\ell)=1.\ee
Moreover, in the finite dimensional Hopf $\C^\star$ case one knows \cite{KLS,NS} that $S^2 =\id$   and there exists a unique two-sided integral $\ell\in H$ satisfying the following properties:
\be 
\label{Haar integrals}
  \ell^2 =\ell,\quad \ell^\star =\ell,\quad S(\ell)=\ell,\quad \ell \in \text{Cocom ($H$)}.
  \ee
Here Cocom($H$) means the flipped/opposite coproduct and the original coproduct of $H$ are the same, i.e., 
$\text{Cocom}(H) = \lbrace \ell\in H|\Delta^{\text{cop}}(\ell)=\Delta(\ell)\rbrace$.

\medskip

By virtue of duality, $H^{*}$ is also a finite-dimensional  Hopf $C^{\star}$-algebra.
We define the 
 extended Hilbert  space  $\mathcal{H}_{{\Gamma}}$ for the model by 
\be
\label{hilbert} \mathcal{H}_{{\Gamma}}=\bigotimes_{e\in \Gamma} H^{*},
\ee
where we define the tensor product inner product structure. Here the non-degenerate Hermitian inner product on each $H^{*}$ is defined by \cite{KLS,NS}
\begin{equation}
\label{inner product}
\langle \phi|\psi\rangle_{H^{*}} := \<\ell, \phi^{\star}\psi\>, \hspace{0.4cm}\phi,\psi\in H^{*},
\end{equation}
where $\ell$ is the normalized Haar integral of $H$.

\medskip

The inner product \eqref{inner product} allows to define the adjoint of  the triangle operators $L_{\pm}$ and $T_{\pm}$ 
\[ (L_{\pm}^{h})^{\dagger}=L_{\pm}^{h^{\star}},\qquad (T_{\pm}^{a})^{\dagger}=T_{\pm}^{a^{\star}}. \]
For example, we check this for $T_{-}^{a}$ as follows:
\bes
\langle \phi|T^{a}_{-}(\psi)\rangle_{H^{*}} &=& \langle \ell,\phi^{\star}T^{a}_{-}(\psi)\rangle = \langle \ell,\phi^{\star}\langle a,\psi\t\rangle\psi\o\rangle =\langle \ell,\langle a\thr,\psi\t\rangle\langle a\t Sa\o,\phi\t^{*}\rangle\phi\o^{*}\psi\o\rangle\nonumber\\
&=& \langle \ell,\langle a\thr,\psi\t\rangle\langle a\t,\phi\t^{*}\rangle\langle Sa\o,\phi\thr^{*}\rangle\phi\o^{*}\psi\o\rangle\nonumber\\
&=& \langle \ell,\langle a\t,\psi\t\phi\t^{*}\rangle\langle Sa\o,\phi\thr^{*}\rangle\phi\o^{*}\psi\o\rangle\nonumber\\
&=& \langle \ell,\langle a\t,(\psi\phi\o^{*})\t\rangle\langle Sa\o,\phi\t^{*}\rangle(\phi\o^{*}\psi)\o\rangle\nonumber\\
&=& \langle \ell,\langle a\t,\phi\t^{*}\rangle\phi\o^{*}\psi\rangle = \langle T_{-}^{a^{*}}(\phi)|\psi\rangle.\nonumber
\ees
Similarly this holds  for $L_\pm$ and $T_+$.
Consequently,  the adjoint operators of $A^{h}$ and $B^{a}$  are respectively
\begin{equation}
\label{Hs3}
(A^h(v,p))^{\dagger} = A^{h^{\star}}(v,p),\quad 
(B^a(v,p))^{\dagger} = B^{a^{\star}}(v,p).
\end{equation}

We intend to build the Hamiltonian from the geometric operators. By construction, we want an Hamiltonian which is hermitian. Hence we would like to restrict the geometric operators so that they become hermitian. 

One way to proceed is to use the Haar integral which will have two main advantages. First  by construction $\ell^\star=\ell$ so it will ensure that the operators are hermitian. Second, because $\ell^2=\ell$, the geometric operators will become projectors. 
The choice of the Haar integral is in fact   key in proving that the model will be topological, in the quantum double case \cite{BMCA}. 
%
\begin{lemma}
\label{lemmaproj}\textit{ Let  $\ell\in H$, $k\in H^{\text{cop}}$ be normalized Haar integrals of the finite-dimensional Hopf $C^{\star}$-algebras $H$ and $H^{\text{cop}}$ respectively. The vertex and face operators 
$A^\ell(v,p), B^k(v,p) :H^{*\tens |E|}\rightarrow H^{*\tens |E|}$ form a set of commuting Hermitian projectors. 
}
\end{lemma}
\begin{proof}
The proof follows directly from the properties of the Haar integral outlined in  \eqref{Haar integrals}.
\[
(A^{\ell}(v,p))^2= A^{\ell^2}(v,p)=A^{\ell}(v,p)\equiv A_v, \quad (B^{k}(v,p))^2 = B^{k^2}(v,p)=B^{k}(v,p) .
\]
From \eqref{AB-bicross}, using the properties of the Haar integrals \eqref{Haar integrals}, it is clear that these projectors commute. 
\end{proof}
The vertex and face operators  depend on the cyclic ordering of the edge ends at each vertex but no longer on the starting point one has to make. This is because the Haar integrals  have the $\text{$n$-th}$ coproduct $\Delta^{(n)}(\ell)$ invariant \cite{CM} under cyclic permutations $\forall n\in\mathbb{N}$. 
Hence we can now shorten the notation
\be
A^{\ell}(v,p)\equiv A_v, \quad B^{k}(v,p)\equiv B_p.
\ee

We are now ready to  define the Hamiltonian of the theory. 
\begin{defn}\label{HamiltonianPS}
\textit{ 
 For   $\Gamma$ as above, the Hamiltonian defining  the bicrossproduct  $M(H)$ model on $H_\Gamma$  is given by
\be\label{genham}
\mH= \sum_{v\in V}(\id- A_v) +\sum_{p\in F}(\id -B_p).
\ee
 The space of ground states or protected space of the Hamiltonian \eqref{genham} is given by 
the invariant  subspace $\cP_\Gamma$ of $\mH$:
\begin{equation}\label{ground state}
\cP_\Gamma :=\lbrace \phi\in\CH_{{\Gamma}}:A_v(\phi)=\phi,\;B_p(\phi)=\phi, \; \forall v, p\rbrace.
\end{equation} 
}
\end{defn}
Since the operators $A_v$ and $ B_p$ are self-adjoint, one ensures that  the Hamiltonian is self-adjoint.



\section{Tensor network representations for semidual models}
\label{Tnetwork}
We would like now to determine a representation of the ground state $\cP_\Gamma$ of $\mH$.   
Following closely \cite{BMCA}, we construct a tensor network representation  of the ground state, parametrized by a graph $\Gamma$, of the mirror bicrossproduct  model of Section \ref{BCPmodel}.  Such representation can be useful to probe the entanglement properties of the states and define some hierarchy of states as discussed in  \cite{BMCA}.

Our starting point is to provide the diagrammatic framework for the tensor network states built on $\Gamma$ and decorated by $H^*$. 
The construction includes graphs whose underlying surface has boundaries.

\subsection{Diagrammatic scheme for tensor network states and tensor trace}

\paragraph{Graph.} The set of edges $E$ of the graph $\Gamma$ corresponding to the surface $\Sigma$ may be decomposed into a disjoint union of interior edges 
and boundary edges. 
If due to the presence of the boundary, we have open edges, i.e., not  closed faces, we can deform the graph in the appropriate way to have closed faces, as depicted in   Figure \ref{boundary faces}, which we reproduce\footnote{This is Figure 4 from \cite{BMCA}.} from \cite{BMCA}. 
%

\begin{figure}[H]
\begin{center}
\includegraphics[scale=.8]{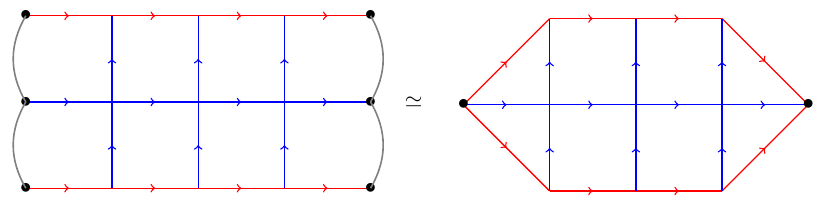}
\caption{The interior and boundary faces of a graph. Figure reproduced from \cite{BMCA}. The key-point is that some boundary faces might not be plain square plaquettes. 
} \label{boundary faces}
\end{center}
\end{figure}

Forming a set consisting of the different sections of the boundary edges and faces, a natural ordering is inherited from the orientation of the boundary of the surface. 
For our discussion, the orientation of any boundary is fixed at an anticlockwise direction with regards to the interior of the surface. 

\smallskip

\paragraph{Tensor of rank 2.} 
To  each pair of oriented edge  $e$  and face $p$, we  associate a tensor of rank 2, as indicated in Figure \ref{tensor}.
\begin{figure}[H]
\begin{center}
\includegraphics{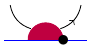}
\caption{ The black dot represents the orientation  of the edge inherited from the underlying graph. 
 The black arrows (virtual edges) attached to the tensor represents the indices of the tensor. 
The association of a tensor to each edge of $\Gamma$ also amounts to placing an anti-clockwise oriented virtual loop in each face  of the graph ${\Gamma}$. A virtual loop determines a face $p\in F$. This tensor has two legs hence it is seen as  a tensor of rank 2.}
\label{tensor}
\end{center}
\end{figure}
Rank 2 tensors can be tensored to generate higher rank tensors and then contracted to  follow the combinatorics of the graph $\Gamma$. They form in this case a \textit{tensor network}, which is then a (complex) number since all legs are contracted, such as in Fig. \ref{Full-tensor1}. 
\begin{figure}[H]
\begin{center}
\includegraphics[scale=.8]{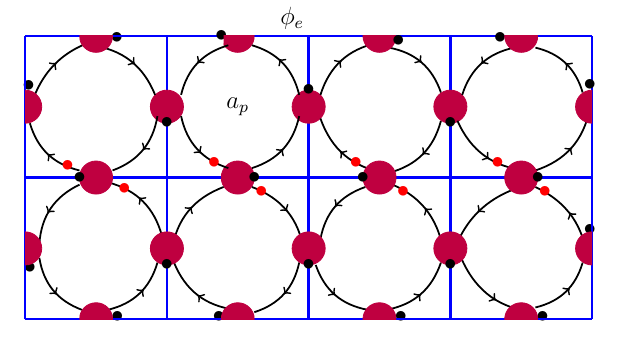}
\caption{
Example of a tensor network for a  graph $\Gamma$. The reason of the presence of the red bullets will become clear in Figure \ref{virtual glueing}. }\label{Full-tensor1}
\end{center}
\end{figure}
The rule for contraction of the tensor network is as follows: one first splits the tensor in terms of rank 2 tensors and then contracts each pair of virtual edges separately and  glue these pieces together.  

\paragraph{Trace.}
To associate a number to a given tensor network, we use the fact that an element $\phi_{e}\in H^*$ is associated to each edge $e\in E$ and that we can associate an element $a_{p}\in H^{\text{cop}}$ to each plaquette.  We can then use the canonical pairing to associate a number to a (degenerate) decorated graph which we interpret as a trace of the tensor, as illustrated in Fig. \ref{pairingcontraction}, for a tensor of rank 2. 
\begin{figure}[H]
\begin{center}
\includegraphics{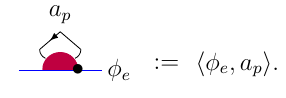}
\caption{We consider first the  simplest tensor depending on a single edge (decorated by $\phi_e$) bounding a single face (decorated by $a_p$). This can be viewed as a tensor of rank 2. The trace of the tensor is obtained by forming the virtual loop and using the canonical pairing between $H^*$ and $H^{cop}$. Geometrically this can be interpreted as the degenerate graph of a single edge bounding a single face.  
    }
\label{pairingcontraction}
\end{center}
\end{figure}
Different orientations of graph edges and  loops are related using the relevant antipode as shown in Figure \ref{edge orientation}.
\begin{figure}[H]
\begin{center}
\includegraphics{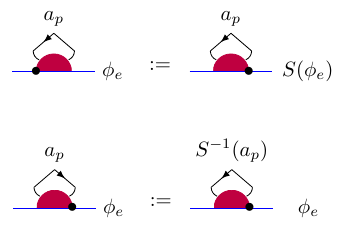}
\caption{The different antipodes are used to change the orientation. These diagrams are assumed to be invariant under global rotations in the plane. 
} \label{edge orientation}
\end{center}
\end{figure}

\paragraph{Contraction.} 
Note that while we contract tensors, it is convenient to also take the trace since at the end we are interested in the traced tensor (ie the tensor network). 
If the face $p$ has more edges $e$ in its boundary, we extend the diagrams in Figure \ref{pairingcontraction} as follows. Consider another edge $e'$ which shares a common vertex with $e$. We define a  gluing operation as in Figure \ref{virtual glueing}. Here the arrows indicate the order in which the coproduct of $a_{p}$ is applied to the basic diagrams. The red dot indicates the origin of this coproduct.
\begin{figure}[H]
\begin{center}
\includegraphics{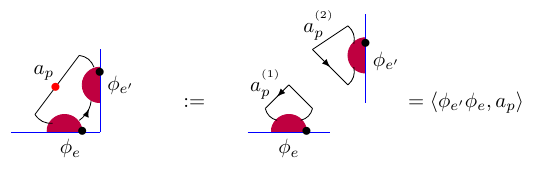}
\caption{Gluing edges with the coproduct of $H^{\text{cop}}$.} \label{virtual glueing}
\end{center}
\end{figure}
We can keep adding edges to form a plaquette and contract the associated rank 2 tensors. Taking the trace for the final tensor is equivalent to close the virtual loop. The value of this trace is obtained by taking the  product of the elements $\phi_i\in H^*$ associated to the edges. 
\begin{figure}[H]
\begin{center}
\includegraphics{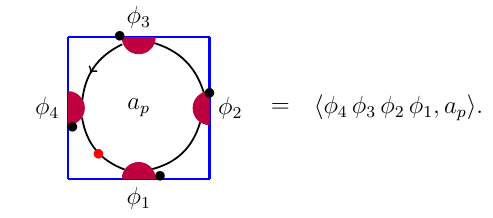}
\caption{We keep contracting tensors of rank 2 and trace the final contraction to obtain the virtual loop and the associated number.} \label{plaquettetrace}
\end{center}
\end{figure}
We note that the (red) dot on the virtual loop is important unless we deal with a co-commutative algebra $H^{\text{cop}}$.  It indicates the origin of the product  in $H^*$ which is constructed bearing in mind the $cop$ in $H^{\text{cop}}$.  Furthermore, we also note that taking the trace of what looked like a degenerate graph in Figure \ref{tensor} can also be interpreted as  the trace of the tensor associated to the face $p$ if we have $\phi_e\equiv \phi_4\cdots \phi_1$. Hence as a shorthand notation for a given edge bounding a face we will sometimes use the shorthand picture displayed in Figure \ref{notation1}
\begin{figure}[H]
\begin{center}\includegraphics{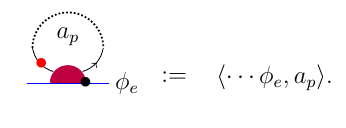}
\caption{For a given face bounded by at most to four edges, we determine the trace of the associated rank 2 tensor. The dashed line represents the possible other red disc contributions associated to the other edges  bounding the face. The $\cdots$ in the formula encodes the elements $\phi_i\in H^*$ associated to the edges bounding the face $p$.  The red dot indicates that the edge we represented is the first   in the (opposite) product.  } \label{notation1}
\end{center}
\end{figure}  
  

\paragraph{Increasing the tensor rank from 2 to 4.}
The edge $e$ will be in general adjacent to two faces, if there are no boundary.  Hence for any  edge $e$ with adjacent  faces $p,q$, we pick $\phi_{e}\in H^{*}$ and $a_{p},a_{q}\in H^{\text{cop}}$ and define the face  gluing operation as in Figure \ref{physical glueing}. 

\begin{figure}[H]
\begin{center}\includegraphics{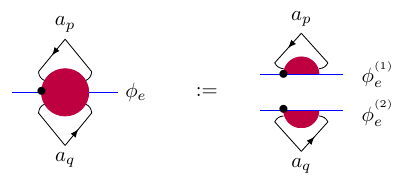}
\caption{Gluing the faces with the coproduct of $H^*$, $\Delta \phi_e=\phi_e^{(1)}\otimes \phi_e^{(2)}$. } \label{physical glueing}
\end{center}\end{figure}
When the faces $p$ and $q$ have many edges, we have to put together Figure \ref{virtual glueing} and Figure \ref{physical glueing}. If furthermore one loop is outgoing we have to also consider the antipode following Figure \ref{edge orientation}.
%
%
%

To evaluate a pair of  virtual loop $p$ and $q$, one performs a clockwise multiplication of all elements labeling the physical edges of the loop and canonically pair the   result  with $a_{p}$ and $a_q$.  Graphically this is illustrated in Figure \ref{contraction rule1}. 
 \begin{figure}[H]
\begin{center}
\includegraphics{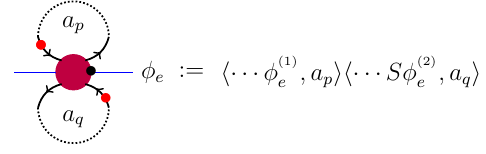}
\caption{We denote by a dashed line the other tensor contribution for the other edges that bound the plaquette $p$. The $\cdots$ in the tensor trace denotes the contribution of these edges.  
 } \label{contraction rule1}
\end{center}
\end{figure}

A change in the orientation of a physical edge using   the antipode in $H^*$  changes the orientation of the corresponding tensor   as shown in  \eqref{edgeorientation}.
\begin{equation}
\includegraphics{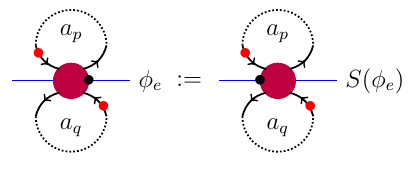}
\label{edgeorientation}
\end{equation}

\paragraph{Trace of a tensor of arbitrary rank.} The  construction of the trace extends to a general tensor. 

\begin{defn}
\label{Hopf trace with boundaries}
\textit{ (Hopf tensor trace with boundaries) 
Let $\partial E$ and $\partial F$ be sets of boundary edges and faces respectively of $\Gamma$.
The Hopf tensor trace associated with the graph ${\Gamma}$ is the function
$\text{ttr}_{{\Gamma}}:
 H^{*\otimes |E|} \otimes H^{\text{cop}\otimes |F|} \otimes H^{*\otimes |\partial E|} \otimes H^{\text{cop}\otimes |\partial F|}\rightarrow \mathbb{C}$,
\begin{equation}
\bigotimes_{e\in E}\phi_{e}\bigotimes_{p\in F}a_{p} \bigotimes_{e\in\partial E}\phi'_{e} \bigotimes_{q\in \partial F}a_{q}\longmapsto \text{ttr}_{{\Gamma}}(\lbrace \phi_{e}\rbrace ;\lbrace a_{p}\rbrace;\lbrace \phi'_{e}\rbrace;\lbrace a_{q}\rbrace  )
\end{equation}
which is defined via  diagrams and the evaluation rules given in Figures  \ref{virtual glueing}, \ref{physical glueing} and \ref{contraction rule1}.  
}
\end{defn}
We emphasize again that our construction is very similar to the one of \cite{BMCA}. Note that the  Hopf tensor trace in the bicrossproduct model  acts on a space  dual to that of the quantum double model  defined in \cite{BMCA},  $(H\tens H^{*\text{op}})^{*}= H^{*}\tens H^{\text{cop}}$.

\subsection{Quantum state}

We now use the tensor trace to define quantum states  for the bicrossproduct model. For simplicity we will focus on the no-boundary case. 
\begin{defn}
\textit{ Let $\phi_{e}\in H^{*}$ and $a_{p}\in H^{\text{cop}}$.  Let $\Gamma$  be  the graph embedded in a surface $\Sigma$  with no boundaries. The Hopf tensor network state on  $\Gamma$ is given by
\begin{equation}
|\Psi_{{\Gamma}}(\lbrace \phi_{e}\rbrace ;\lbrace a_{p}\rbrace)\rangle := 
\text{ttr}_{{\Gamma}}(\lbrace \phi_{e}^{\sst}\rbrace ;\lbrace a_{p}\rbrace)\bigotimes_{e\in E
}|\phi_{e}^{\sso}\rangle.
\end{equation}
}
\end{defn}

 We shall now proceed to  solve the bicrossproduct model in this framework of Hopf tensor network states.
We  choose a particular Hopf tensor network state as a ground state of the model so that this state is  topologically ordered.  

\begin{pro}
(Ground state of the mirror bicrossproduct model).\textit{
Let $\eta$ and $k$ the Haar integrals of $H^{*}$ and $H^{\text{cop}}$ respectively. Then a ground state $|\Psi_{{\Gamma}}\rangle$ of the mirror bicrossproduct 
 model parametrized by $\Gamma$,   $|\Psi_{{\Gamma}}\rangle \in \cP_\Gamma\subset \mH$ is 
\begin{equation}
|\Psi_{{\Gamma}}\rangle :=|\Psi_{{\Gamma}}(\lbrace \eta_{e}\rbrace;\lbrace k_{p}\rbrace)\rangle.
\end{equation}}
\end{pro}
\begin{proof}
Recall that  the Hamiltonian for the mirror bicrossproduct model is a sum of local commuting terms  $A_{v}$ and $B_{p}$. Hence it is sufficient to show that the operators $A_{v}$ and $B_{p}$ leave the state $|\Psi_{{\Gamma}}\rangle$ invariant individually. To prove the proposition, it is enough to focus on the action of such operators on either a face or a vertex of the graph $\Gamma$. 

\medskip 

A face $p$ of  in the bulk of $\Gamma$ decorated by the Haar integral $k$, bounded by the four edges decorated by the Haar integral  $\eta^i\in H^{*}$ and the four faces decorated by   Haar integral $a^i$ leads to the contribution  $|\Psi_{p}\rangle$  given by  Fig. \ref{facecontrib}. Other edge/face orientations than in Fig. \ref{facecontrib} can be treated likewise. 
\begin{figure}[H]
 \begin{center}
\includegraphics[scale=0.9]{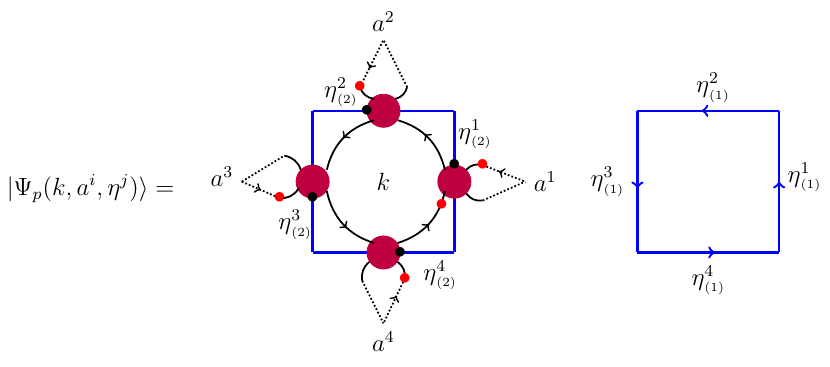}
\caption{Contribution associated to the face $p$ to the state $\Psi_\Gamma$.  Note that the first  diagram on the right side above is the tensor trace function. } \label{facecontrib}
\end{center}
\end{figure}
The state $|\Psi_{p}(k,a^{i},\eta^j)\rangle$ written in an  explicit form reads 
\begin{eqnarray}
\label{interior face}
|\Psi_{p}(k,a^{i},\eta^j)\rangle &=&\langle\eta_{\t\t}^{4}\cdots \eta_{\t\t}^{1},\,k\rangle\,\prod_{j=1}^{4}\langle \cdots S\eta_{\t\o}^{j},\,a^{j}\rangle
  |\eta_{\o}^{1}\rangle\otimes\cdots \otimes|\eta_{\o}^{4}\rangle\nonumber\\
&=& \langle\eta_{\thr}^{4}\cdots \eta_{\thr}^{1},\,k\rangle\prod_{j=1}^{4}\langle \cdots S\eta_{\t}^{j} ,\,a^{j}\rangle\,|\eta_{\o}^{1}\rangle\otimes\cdots  \otimes|\eta_{\o}^{4}\rangle,
\end{eqnarray}
where 
the expression $\eta_{\thr}^{4}\cdots \eta_{\thr}^{1}$ denote composition of the elements $\eta$.  We note that in the case there would be no  face $j$ for example bounding $p$, we would just delete the associated $\langle S\eta_{\t}^{j},\,a^{j}\rangle$ contribution.

The state \eqref{interior face} is invariant under the action of $B_{p}$ as shown below
\begin{eqnarray}
\label{face invariant}
B^{k}_{p}|\Psi_{p}(k,a^{i},\eta^j)\rangle &=& \langle\eta_{\thr}^{4}\cdots \eta_{\thr}^{1},\,k\rangle\langle\eta_{\o\o}^{4}\cdots \eta_{\o\o}^{1},\,k\rangle \prod_{j=1}^{4}\langle \cdots S\eta_{\t}^{j},\,a^{j}\rangle\,|\eta_{\o\t}^{1}\rangle\otimes\cdots\otimes |\eta_{\o\t}^{4}\rangle\nonumber\\
 &=& \langle\eta_{\fo}^{4}\cdots \eta_{\fo}^{1},\,k\rangle\langle\eta_{\o}^{4}\cdots \eta_{\o}^{1},\,k\rangle \prod_{j=1}^{4} \langle \cdots S\eta_{\thr}^{j},\,a^{j}\rangle\,|\eta_{\t}^{1}\rangle\otimes\cdots\otimes |\eta_{\t}^{4}\rangle\nonumber\\
&=& \langle\eta_{\thr}^{4}\cdots \eta_{\thr}^{1},\,k\rangle\langle\eta_{\fo}^{4}\cdots \eta_{\fo}^{1},\,k\rangle \prod_{j=1}^{4}\langle \cdots S\eta_{\t}^{j},\,a^{j}\rangle\,|\eta_{\o}^{1}\rangle\otimes\cdots\otimes |\eta_{\o}^{4}\rangle\nonumber\\
&=& \langle(\eta_{\thr}^{4}\cdots \eta_{\thr}^{1})^2,\,k\rangle\prod_{j=1}^{4} \langle \cdots S\eta_{\t}^{j},\,a^{j}\rangle\,|\eta_{\o}^{1}\rangle\otimes\cdots\otimes |\eta_{\o}^{4}\rangle\nonumber\\
&=& \langle\eta_{\thr}^{4}\cdots \eta_{\thr}^{1},\,k\rangle\prod_{j}^{4} \langle \cdots S\eta_{\t}^{j},\,a^{j}\rangle\,|\eta_{\o}^{1}\rangle\otimes\cdots\otimes |\eta_{\o}^{4}\rangle\nonumber\\
&=& |\Psi_{p}(k,a^{i},\eta^j)\rangle.
\end{eqnarray} 
In the first equality, we used the action of the face operator $B^{k}_{p}$ of Definition \ref{vex and face operators for arbitrary graph} on the state $|\Psi_{p}(k,a^{i},\eta^j)\rangle$. We perform a renumbering on $\eta$ in the second equality. The fifth equality uses the property of the Haar integral.

\medskip 

Next we consider a vertex $v$ with four ingoing  edges (other orientations can be treated likewise).  Again, $\eta^j$ and $a^i$ are Haar integrals. The vertex contribution to $\Psi_{{\Gamma}}$ is given by Fig. \ref{contribvertex}.

\begin{figure}[H]
\begin{center}
\includegraphics[scale=0.9]{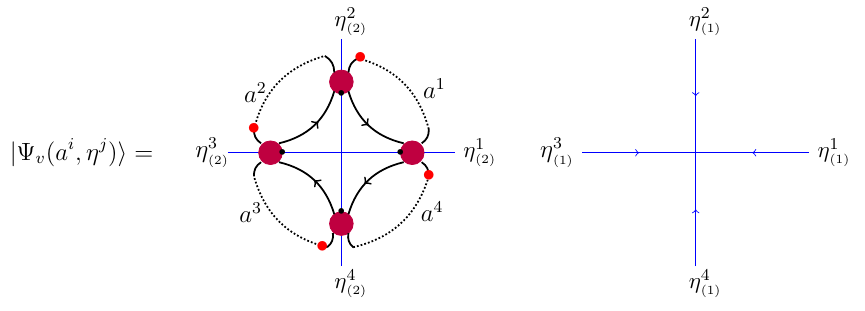}
\caption{Contribution associated to the vertex $v$ to the state $\Psi_\Gamma$.  Note that the first  diagram on the right side above is the tensor trace function. } \label{contribvertex}
\end{center}
\end{figure}

In a more explicit form, the state contribution from the vertex is
\begin{eqnarray}
\label{vertex contribution}
|\Psi_{v}(a^{i},\eta^j)\rangle &=& \langle S\eta^{1}_{\t\o}\cdot \eta^{2}_{\t\t},\,a^{1}\rangle\cdots  
\langle S\eta^{4}_{\t\o}\cdot \eta^{1}_{\t\t},\,a^{4}\rangle\, |\eta_{\o}^{1}\rangle\otimes\cdots \otimes|\eta_{\o}^{4}\rangle\nonumber\\[0.3\baselineskip]
&=& \langle S\eta^{1}_{\t}\cdot \eta^{2}_{\thr},\,a^{1}\rangle\cdots \langle S\eta^{3}_{\t}\cdot \eta^{4}_{\thr},\,a^{4}\rangle\langle S\eta^{4}_{\t}\cdot \eta^{1}_{\thr},\,a^{4}\rangle \, |\eta_{\o}^{1}\rangle\otimes\cdots \otimes|\eta_{\o}^{4}\rangle\nonumber\\
&=& \prod_{j=1}^{4}\langle S\eta^{j}_{\t}\cdot \eta^{j+1}_{\thr},\,a^{j}\rangle |\eta_{\o}^{1}\rangle\otimes\cdots \otimes|\eta_{\o}^{4}\rangle,
\end{eqnarray}
where if $j=4$, $j+1\equiv 1$. This convention will be used below as well. 
This contribution to the ground state  is then invariant under the action of $A_{v}$:
\bes
&&
A^{\ell}_{v}|\Psi_{v}(a^{i},\eta^j)\rangle 
=\prod_{j=1}^{4}\langle S\eta^{j}_{\t}\cdot \eta^{j+1}_{\thr},\,a^{j}\rangle 
\,\,|\eta_{\o\t}^{1}\rangle\otimes\cdots \otimes|\eta_{\o\t}^{4}\rangle\nonumber\\ [0.3\baselineskip]
&& \times\,\langle \ell^j,(S\eta^4_{\o\o})(S\eta^3_{\o\o})(S\eta^2_{\o\o})(S\eta^1_{\o\o})\eta^1_{\o\thr}\eta^2_{\o\thr}\eta^3_{\o\thr}\eta^4_{\o\thr}\rangle \nonumber\\[0.3\baselineskip]
&=& \prod_{j=1}^{4}\langle S\eta^{j}_{\fo}\cdot \eta^{j+1}_{\thr},\,a^{j}\rangle \,\langle \ell^j,(S\eta\o^4)(S\eta\o^3)(S\eta\o^2)(S\eta\o^1)\eta\thr^1\eta\thr^2\eta\thr^3\eta\thr^4\rangle
\,|\eta\t^{1}\rangle\otimes\cdots \otimes|\eta\t^{4}\rangle\, \nonumber\\[0.3\baselineskip]
&=&
\prod_{j=1}^{4}\langle S\eta^{j}_{\fo}\cdot \eta^{j+1}_{\thr},\,a^{j}\rangle \langle\ell\o^j,(S\eta\o^4)(S\eta\o^3)(S\eta\o^2)(S\eta\o^1)\rangle \langle \ell\t^j,\eta\thr^1\eta\thr^2\eta\thr^3\eta\thr^4\rangle
\,|\eta\t^{1}\rangle\otimes\cdots \otimes|\eta\t^{4}\rangle\nonumber
\ees
\bes
&=&
\prod_{j=1}^{4}\langle S\eta^{j}_{\fo}\cdot \eta^{j+1}_{\thr},\,a^{j}\rangle \langle S\ell\o^j,\eta\o^1 \eta\o^2\eta\o^3\eta\o^4\rangle \langle \ell\t^j,\eta\thr^1\eta\thr^2\eta\thr^3\eta\thr^4\rangle
\,|\eta\t^{1}\rangle\otimes\cdots \otimes|\eta\t^{4}\rangle\nonumber\\
&=&
\prod_{j=1}^{4}\langle S\eta^{j}_{\thr}\cdot \eta^{j+1}_{\thr},\,a^{j}\rangle\,\epsilon(\ell^{j})\, \epsilon(\eta^{1}_{\o}) \, \epsilon(\eta^{2}_{\o}) \, \epsilon(\eta^{3}_{\o})\, \epsilon(\eta^{4}_{\o})\,|\eta\t^{1}\rangle\otimes\cdots \otimes|\eta\t^{4}\rangle\nonumber\\
\label{vertex invariant}
&=& 
\prod_{j=1}^{4}\langle S\eta^{j}_{\t}\cdot \eta^{j+1}_{\thr}),\,a^{j} \rangle\,|\eta\o^{1}\rangle\otimes\cdots \otimes|\eta\o^{4}\rangle = |\Psi_{v}(a^{i},\eta^j)\rangle .
\ees
We used the definition of the vertex operator of Definition \ref{vex and face operators for arbitrary graph} in the first equality. We  permute  cyclicly the different components of $\eta$ in the definition of the vertex operator in the second equality. The third equality uses the counit property of a Hopf algebra. While in the fourth equality we used the fact that $\epsilon(\eta^{j}_{\o})=\epsilon(\ell^{j})=1$, to get to the fifth equality. 
\end{proof}
The  quantum state $|\Psi_{{\Gamma}}\rangle$ is nothing but a trivial representation of the $M(H)$ and  a vacuum of the model, parametrized by $\Gamma$. If the model can be shown to be topological, which we expect it to be,  such state should have then a  trivial  topological charge.

\medskip

\section{Outlook}
\label{outlook}
In this article we  proposed for the first time a (Kitaev) lattice model  not based on  the Drinfeld quantum double, but instead on the (mirror) bicrossproduct quantum group. Given a graph with cyclic ordering of edge ends at each vertex, our construction of a Hilbert space for the bicrossproduct model for a Hopf algebra $H$ is based on the extension of the canonical   covariant action of the bicrossproduct quantum group  $M(H)$  on $H^*$ to an action on $H^{*\tens |E|}$, the $|E|$-fold tensor product of $H^*$, where $|E|$ is the number of edges. 
This action which enter the definition of the triangle operators and consequently  the vertex and face operators are in general not required to be covariant as we seek a bicrossproduct module and not a module algebra. {We obtain an exactly solvable Hamiltonian for the new model and also a representation of the ground state by introducing a tensor network representation}.

{As with Kitaev's original construction \cite{AYK}, which exhibited some topological features (such as the degeneracy of the ground state and the anyonic nature of quasi-particles), we suspect our model  also exhibits similar topological properties. These topological features can be further explored thanks to the results in \cite{MO} where the  $R$-matrix -- a structure known to play an important role in the braiding and fusion of quasi-particles -- was explicitly given. We defer the study of the topological invariance (ie the degeneracy of the ground state) and the model  properties to later investigations. }

Furthermore, the physical properties of the ground state have still to be explored. In particular, thanks to the tensor network representation of the ground state we provided, it would be interesting  to see whether we have some area law for the entanglement entropy, or a similar hierarchy as it was  described in \cite{BMCA}. 

\medskip

This new model opens up other new directions to explore. From the quantum gravity perspective, the vertex and face operators are related to the Gauss constraint and the Flatness constraint, which are usually characterized in terms of symmetries by the Drinfeld double in the quantum double model. It would be interesting to determine whether the bicrossproduct case has also some geometrical meaning. The semi-duality between the quantum groups  seems to indicate naively that we dualize somehow for example the Flatness constraint into another Gauss constraint, or vice versa. Investigations are currently underway to see if this argument can be made more rigorous.

\smallskip
As the Kitaev quantum double model is known to be equivalent to  the combinatorial quantization of Chern-Simons theory based on the Drinfeld double \cite{CM}. It would be interesting to see whether this result extends to the bicrossproduct case, namely that our model can be related to the combinatorial quantization of Chern-Simons theory based on the bicrossproduct quantum group.  In the case of the Drinfeld double, one required  a Hopf gauge theoretic  framework \cite{MW}. This provides another interesting question to address in the context of the bicrossproduct model. For this construction, one required a universal  $R$-matrix.  This is now known explicitly for the bicrossproduct quantum group due  to  recent work in \cite{MO} which provides an explicit expression of the $R$-matrix for this  quantum group.

\acknowledgments
 The authors would like to thank T. Fritz for his comments and questions. We also thank the anonymous referee for many constructive suggestions and in particular for  identifying some key issues in the previous drafts. 
 
 This research was partly  supported by Perimeter Institute for Theoretical Physics. Research
at Perimeter Institute is supported by the Government of Canada through the Department of Innovation, Science and
Economic Development Canada and by the Province of Ontario through the Ministry of Research, Innovation and
Science.

\bibliographystyle{JHEP}

\bibliography{biblio-kitaev}

\end{document}